\documentclass[english]{article}

\usepackage[latin1]{inputenc}
\usepackage{tipa}
\usepackage{dsfont}
 
\usepackage{babel}
\usepackage{leftidx}

\usepackage{xr}
 
\usepackage{amsxtra}
\usepackage{amsmath}
\usepackage{amssymb}
\usepackage{amsfonts}
\usepackage[all,2cell]{xy}
\UseAllTwocells
\usepackage{mathrsfs}
\usepackage{amsthm}
\usepackage{enumitem}
\usepackage{hyperref}
\usepackage{subfigure}
\usepackage{todonotes}

\usepackage{euscript}

\setlist[enumerate]{leftmargin=*,labelindent=.5pc}
\usepackage{fullpage}

\newtheorem{thm}{Theorem}
\newtheorem{cor}[thm]{Corollary}

\newtheorem{lem}[thm]{Lemma}
\newtheorem{prop}[thm]{Proposition}

\newtheoremstyle{example}{\topsep}{\topsep}%
     {}
     {}
     {\bfseries}
     {.}
     {2pt}
     {\thmname{#1}\thmnumber{ #2}\thmnote{ #3}}

   \theoremstyle{example}
\newtheorem{defi}[equation]{Definition}
\newtheorem{rem}[equation]{Remark}

\newtheorem{exa}[equation]{Example}
\newtheorem{ex}[equation]{Example}

\def\on{\operatorname}
\def\op{{\operatorname{op}}}

\def\Fun{\operatorname{Fun}}
\def\UFun{\underline{\Fun}}

\def\colim{\operatorname*{colim}}

\def\Hom{\operatorname{Hom}}

\def\Aut{\operatorname{Aut}}
\def\S{\EuScript S}

\def\O{{\mathcal O}}

\def\fC{ {\mathfrak C}}

\def\Map{\operatorname{Map}}

\def\N{\operatorname{N}}

\def\cm{\langle m \rangle}
\def\cn{\langle n \rangle}

\def\T{{\mathcal T}}

\def\Set{ {\mathcal Set}}

\def\cone{\operatorname{cone}}
\def\fib{\operatorname{fib}}
\def\cof{\operatorname{cof}}

 \def\dg{\on{dg}}
 \def\Sing{\on{Sing}}
 
\def\Ndg{\N_{\on{dg}}}

\def\A{ {\EuScript A}}
\def\X{ {\EuScript X}}
\def\Y{ {\EuScript Y}}
\def\Z{ {\EuScript Z}}
\def\Ae{ \A^e}
\def\Se{ \S^e}
\def\Pe{ \P^e}
\def\AP{ \A_\P}
\def\Be{ \B^e}
\def\pt{ \on{pt}}
\def\AA{ {\mathbb A}}

\def\bC{ {\mathbb C}}
\def\PP{ {\mathbb P}}
\def\Ind{\on{IndCoh}}
\def\QCoh{\on{QCoh}}
\def\coh{ \on{coh}}
\def\Coh{ \on{Coh}}
\def\Fuk{ \on{Fuk}}
\def\P{{\EuScript P}}
\def\B{ {\EuScript B}}

\def\C{ {\EuScript C}}
\def\Ho{ {\on{Ho}}}

\def\D{ {\EuScript D}}

\def\M{{\EuScript M}}

\def\O{{\mathcal O}}

\def\om{{\omega}}

\def\Set{ {\mathcal Set}}

\def\<<{\langle {}\hskip -.1cm {}\langle}
\def\>>{\rangle \hskip -.1cm \rangle}

\def\one{ {\bf 1}}

\setlist[enumerate,1]{label=(\arabic{*})}
\setlist[enumerate,2]{label=(\alph{*})}
\setlist[enumerate,3]{label=(\roman{*})}

\def\centerarc[#1](#2)(#3:#4:#5)  { \draw[#1] ($(#2)+({#5*cos(#3)},{#5*sin(#3)})$) arc (#3:#4:#5); }
\tikzstyle{dot}=[draw,circle,fill=black,inner sep=0,minimum size=4pt]
\tikzset{
    partial ellipse/.style args={#1:#2:#3}{
        insert path={+ (#1:#3) arc (#1:#2:#3)}
    }
}

\def\lra{\longrightarrow}
\def\llra{\longleftrightarrow}

\def\ZZ{\mathbb{Z}}

\def\Cat{{\EuScript Cat}}
\def\Catdg{{\Cat}_{\dg}}
\def\Lqe{\on{L}_{\on{eq}}(\Catdg(k))}
\def\Lmo{\on{L}_{\on{mo}}(\Catdg(k))}

\def\RHom{R\underline{\on{Hom}}}
\def\UHom{\underline{\on{Hom}}}

\def\Mod{\on{Mod}}
\def\Ch{\on{Ch}}

\def\Perf{\on{Perf}}

\def\MF{\on{MF}}

\def\T{ {\EuScript T}}

\def\L{ {\EuScript L}}

\def\id{\on{id}}

\def\cm{\langle m \rangle}
\def\cn{\langle n \rangle}

\def\presuper#1#2%
  {\mathop{}%
   \mathopen{\vphantom{#2}}^{#1}%
   \kern-\scriptspace%
   #2}
\def\presub#1#2%
  {\mathop{}%
   \mathopen{\vphantom{#2}}_{#1}%
   \kern-\scriptspace%
   #2}
   \def\PA{ \presub{\P}{\A}}

   \def\HH{\on{HH}}
   \def\CB#1{ {{\on{CC}_{\bullet}(#1)}}}
   \def\CC#1{ {{\on{CC}_{\bullet}(#1)}_{S^1}}}
   \def\NC#1{ {{\on{CC}_{\bullet}(#1)}^{S^1}}}

\title{Relative Calabi-Yau structures}
\author{Christopher Brav\footnote{Higher School of Economics, Moscow, email:{\tt c.brav@hse.ru}}
\;and Tobias Dyckerhoff\footnote{Hausdorff Center for Mathematics, Bonn University, email:{\tt dyckerho@math.uni-bonn.de}}}

\begin{document}
\maketitle
\begin{abstract}
	We introduce relative noncommutative Calabi-Yau structures defined on functors of
	differential graded categories. Examples arise in various contexts such as topology,
	algebraic geometry, and representation theory. Our main result is a composition law for
	Calabi-Yau cospans generalizing the classical composition of cobordisms of oriented
	manifolds. As an application, we construct Calabi-Yau structures on topological Fukaya
	categories of framed punctured Riemann surfaces.
\end{abstract}

\tableofcontents

\vfill\eject

\numberwithin{equation}{section}
\numberwithin{thm}{section}

\newpage

\section{Introduction}

Let $X$ be a topological space and let $k$ be a field. Consider the Quillen adjunction 
\begin{equation}\label{eq:quillen}
		\dg: \Set_{\Delta} \llra \Catdg(k): \Ndg
\end{equation}
between quasi-categories and $k$-linear differential graded categories given by the dg nerve
construction (cf. \cite{lurie:algebra}). We may use this adjunction to linearize the
space $X$, defining the dg category
\[
	\L(X) := \dg(\Sing(X))
\]
where $\Sing(X)$ denotes the singular simplicial set of $X$. If $X$ is connected, then $\L(X)$ is
quasi-equivalent to the dg algebra of chains on the based loop space of $X$. 
The adjunction \eqref{eq:quillen} induces a weak equivalence of spaces
\[
	\Map(\Sing(X), \Ndg(\Ch(k))) \simeq \Map(\L(X), \Ch(k))
\]
which allows us to interpret dg modules over $\L(X)$ as $\infty$-local systems of
cochain complexes of $k$-vector spaces on $X$.
Assume now that $X$ is the topological space underlying a closed oriented manifold. Then the choice of orientation on $X$
is reflected in a Calabi-Yau structure on $\L(X)$ in the following sense:

\begin{defi}\label{defi:lcy}
	Let $\A$ be a smooth $k$-linear dg category (i.e. $\A$ is Morita equivalent to a dg algebra which
	is perfect as a bimodule). An $n$-dimensional {\em left Calabi-Yau
	structure} on $\A$ consists of a cycle
	\begin{equation}\label{eq:equivariance}
		\widetilde{[\A]}: k[n] \lra \NC{\A}
	\end{equation}
	in the negative cyclic complex such that the induced morphism of $\A$-bimodules
	\begin{equation}\label{eq:orientation}
			\A^! \lra \A[-n]
	\end{equation}
	is a quasi-isomorphism where $\A^!$ denotes the derived $\A^{\op} \otimes \A$-linear dual of
	$\A$. 
\end{defi}

The smoothness hypothesis implies that $\A^!$ is the Morita-theoretic {\em left} adjoint of the evaluation
functor
\[
	\A^{\op} \otimes \A \to \Ch(k), (x,y) \mapsto \A(x,y)
\]
which explains our choice of terminology.
The idea to formulate Calabi-Yau structures in terms of the quasi-isomorphism \eqref{eq:orientation}
has appeared in \cite{ginzburg:cy}. The additional $S^1$-equivariance data \eqref{eq:equivariance}
has been proposed by Kontsevich-Vlassopoulos in \cite{kontsevich-vlassopoulos:weakcy}. The result
that $\L(X)$ carries a canonical left Calabi-Yau structure is stated in \cite{lurie:tqft} and proved
in \cite{cohen-ganatra:string}. 
We recall a different notion of Calabi-Yau structure from \cite{kontsevich-soibelman:notes}:

\begin{defi}
	Let $\A$ be a (locally) proper $k$-linear dg category (i.e. all morphism complexes of $\A$ are perfect).
	An $n$-dimensional {\em right Calabi-Yau structure} on $\A$ consists of a cocycle
	\begin{equation}\label{eq:requivariance}
		\widetilde{\sigma}: \CC{\A} \lra k[-n]
	\end{equation}
	on the cyclic complex so that the induced morphism of $\A$-bimodules
	\begin{equation}\label{eq:rorientation}
		\A[n] \lra \A^*
	\end{equation}
	is a quasi-isomorphism where $\A^*$ denotes the $k$-linear dual of $\A$.
\end{defi}

The properness hypothesis implies that $\A^*$ is the Morita-theoretic {\em right} adjoint of the
evaluation functor $\A^{\op} \otimes \A \to \Ch(k)$ explaining our terminology. 
A left Calabi-Yau structure on a smooth dg category $\A$ induces a right Calabi-Yau 
structure on its Morita dual $\A^{\vee}$. For example, the Morita dual of
$\L(X)$ can be identified with the full dg subcategory of $\Perf_{\L(X)}$ consisting of
$\infty$-local systems of cochain complexes on $X$ with perfect fibers (pseudo-perfect
$\L(X)$-modules in the terminology of \cite{toen-vaquie:moduli}). From the point of view of
noncommutative geometry, as already strongly advocated in \cite{kontsevich:symplectic}, the dg
category $\L(X)$ is much better behaved than its Morita dual $\L(X)^{\vee}$: the fact that $X$ is
homotopy equivalent to a finite CW-complex implies that $\L(X)$ is {\em of finite type} in the sense
of \cite{toen-vaquie:moduli}. Most importantly for us, as shown in loc. cit, its derived moduli
stack of pseudo-perfect modules is locally geometric and of finite presentation so that it has a
reasonable cotangent complex. This is in stark contrast to the Morita dual $\L(X)^{\vee}$ which is
(locally) proper but almost never of finite type so that a reasonable moduli in the sense of
\cite{toen-vaquie:moduli} does not exist. It is for these reasons that, in the context of most
examples appearing in this work, we are led to consider left Calabi-Yau structures as more
fundamental than right Calabi-Yau structures.

Now assume that $X$ is an oriented compact manifold but with possibly nonempty boundary
$\partial X$. We obtain a corresponding dg functor
\[
	\L(\partial X) \lra \L(X)
\]
of linearizations. To capture the orientation of $(X, \partial X)$ in terms of this functor, we
propose the following relative version of Definition \ref{defi:lcy}:

\begin{defi}\label{defi:rlcy} An $n$-dimensional {\em relative left Calabi-Yau structure} on a
	functor $f: \A \to \B$ of smooth $k$-linear dg categories consists of a cycle
	\begin{equation}\label{eq:rel-equivariance}
			\widetilde{[\B,\A]}: k[n] \lra  \NC{\B,\A}
	\end{equation}
	in the relative negative cyclic complex such that all vertical morphisms in the induced diagram of $\B$-bimodules
	\begin{equation}\label{eq:nondegenerate}
		\xymatrix@C=10ex{
			\B^! \ar[r]^-{c^!} \ar@{-->}[d] & (\B \otimes_{\A}^L \B)^!
			\ar[d] \ar[r] & \cone(c^!) 
			\ar@{-->}[d]\\
			\cone(c)[-n] \ar[r] & (\B \otimes_{\A}^L \B)[-n+1] \ar[r]^-{c[-n+1]} & \B[-n+1]
		}
	\end{equation}
	are quasi-isomorphisms where the morphism $c$ represents the counit 
	of the derived Morita-adjunction
	\begin{equation}\label{eq:morita}
			Lf_!: \D(\Mod_{\A}) \llra \D(\Mod_{\B}): Rf^*.
	\end{equation}
\end{defi}

A relative left Calabi-Yau structure on the zero functor $0 \to \B$ can be identified with
an absolute left Calabi-Yau structure on $\B$. Diagram \eqref{eq:nondegenerate} admits a 
particularly nice interpretation if we assume the dg functor $f$ to be spherical in the
sense of \cite{anno-logvinenko:spherical} (this
is satisfied in many of our examples): in this case, the $\B$-bimodule $\cone(c)$ represents an
autoequivalence of $\Mod_{\B}$ known as a {\em spherical twist} of the adjunction. Therefore, in
this context, Diagram \eqref{eq:nondegenerate} amounts to an identification of the inverse dualizing
bimodule $\B^!$ with a shifted spherical twist of the adjunction \eqref{eq:morita}. There are also
relative variants of right Calabi-Yau structures (these have already been introduced in 
\cite[5.3]{toen:derived}) and, generalizing the absolute case, the two notions are related via
Morita duality.

A basic operation in cobordism theory is to glue two manifolds along a common boundary
component to produce a new manifold: given manifolds $X$ and $X'$ with boundary decompositions $Z \amalg Z' = \partial X$
and $Z' \amalg Z'' = \partial X'$, we have a commutative diagram
\[
		\xymatrix{
			& & Z \ar[d]  \\
			& Z' \ar[r] \ar[d] & X \ar[d] \\
			Z'' \ar[r] & X' \ar[r] & X \amalg_{Z'} X' }
\]
where the square is a homotopy pushout square of spaces. Linearizing by applying $\L$ yields the commutative diagram 
\[
		\xymatrix{
			& & \L(Z) \ar[d]  \\
			& \L(Z') \ar[r] \ar[d] & \L(X) \ar[d] \\
			\L(Z'') \ar[r] & \L(X') \ar[r] & \L(X \amalg_{Z'} X') }
\]
where the square is a homotopy pushout of dg categories. The fact that composition of cobordisms is
compatible with orientations admits the following noncommutative generalization which is the main
result of this work:

\begin{thm}\label{thm:intro-cobordism} 
Let 
\[
	\A \amalg \A' \lra \B
\]
and 
\[
	\A' \amalg \A'' \lra \B',
\]
be functors of smooth dg categories equipped with relative left Calabi-Yau structures which are
compatible on $\A'$. Then the functor
\[
	\A \amalg \A'' \lra \B \amalg_{\A'} \B'.
\]
inherits a canonical relative left Calabi-Yau structure.
\end{thm}

As an application of this result, we construct relative left Calabi-Yau structures on topological 
Fukaya categories of punctured framed Riemann surfaces.  

We would like to mention that this work has a sequel \cite{bd:moduli} in preparation in which we relate
left Calabi-Yau structures to derived symplectic geometry in the sense of \cite{ptvv}. 
We announce the following main result of that work:

\begin{thm}\label{thm:moduli} Let $k$ be a field of characteristic $0$.
	\begin{enumerate} 
	\item Let $\A$ be a $k$-linear dg category of finite type.  Then an $n$-dimensional left
		Calabi-Yau structure on $\A$ determines a canonical $(2-n)$-shifted symplectic form
		on the derived moduli stack $\M_{\A}$ of pseudo-perfect modules.
	\item Let $f:\A \to \B$ be a functor of $k$-linear dg categories of finite type.  Assume
		that $f$ carries an $(n+1)$-dimensional left Calabi-Yau structure so that the
		corresponding negative cyclic class on $\A$ determines an $n$-dimensional left
		Calabi-Yau structure. Then the induced pullback morphism of derived stacks
		\[
			f^*: \M_{\B} \lra \M_{\A} 
		\]
		carries a canonical Lagrangian structure.
	\end{enumerate}
\end{thm}

Note that this is a variant of a statement announced in \cite[5.3]{toen:derived}. However, the result
stated in loc. cit. uses right Calabi-Yau structures. The use of left Calabi-Yau structures in
Theorem \ref{thm:moduli} allows for applications to finite type categories which are not necessarily
proper. For example, in the context of topological Fukaya categories, Theorem \ref{thm:fukcy} and
Theorem \ref{thm:moduli} imply the following statement:

\begin{thm}\label{thm:fukcymod} Let $(S,M)$ be a stable marked surface with framing on $S \setminus
	M$, and let $F(S,M)$ denote its topological Fukaya category. Then pullback along the
	boundary functor induces a morphism of derived stacks
	\[
		i^*: \M_{F(S,M)} \lra \prod_{\pi_0(\partial S \setminus M)} \M_{\underline{k}}
	\]
	where the right-hand side carries a $2$-shifted symplectic structure and $i^*$ has a
	Lagrangian structure. In particular, if $\partial S$ is empty, then $\M_{F(S,M)}$ has a
	$1$-shifted symplectic structure. 
\end{thm}

Here it is crucial to use versions of topological Fukaya categories which arise as global sections
of cosheaves of dg categories (as opposed to sheaves) since these are of finite type.

We note that similar techniques appear in the recent work \cite{shende-takeda}, however, since the
authors use right Calabi-Yau structures on categories not necessarily of finite type, the relation to
symplectic structures on moduli is unclear to us, for the above mentioned reasons.

In light of Theorem \ref{thm:moduli} one can interpret the theory of absolute (resp. relative) left
Calabi-Yau structures as a noncommutative predual of the geometric theory of shifted symplectic
(resp. Lagrangian) structures. For example, Theorem \ref{thm:intro-cobordism} is a 
predual of \cite[Theorem 4.4]{calaque}.

We provide an outline of the contents of this work. In Section \ref{sec:morita}, we introduce the
technical context of this work: derived Morita theory for dg categories. In Sections
\ref{sec:absolute} and \ref{sec:relative} we give a detailed account of the absolute and relative
Calabi-Yau structures sketched above. Section \ref{sec:examples} provides examples of Calabi-Yau
structures in topology, algebraic geometry and representation theory. In Section \ref{sec:cobordism}
we provide a proof of the main result on the composition of Calabi-Yau cospans.  The final Section
\ref{sec:app} contains the applications to topological Fukaya categories of surfaces.\\

\noindent
{\bf Acknowledgements.} We would like to thank Damien Calaque, Alexander Efimov, Dominic Joyce,
Mikhail Kapranov, Alexander Kuznetsov, Thomas Nikolaus, Peter Teichner, and Bertrand To\"en for
interesting discussions related to the subject of this paper.

\section{Morita theory of differential graded categories}
\label{sec:morita}

We introduce some basic ingredients of Morita theory of differential graded categories which will
form the technical context for this work.

\subsection{Modules over differential graded categories} Let $k$ be a field. We denote by $\Ch(k)$ the category 
of unbounded cochain complexes of vector spaces over $k$ equipped with its usual monoidal structure.
A {\em  differential graded (dg) category} is a category enriched over $\Ch(k)$ equipped with its usual
monoidal structure. We refer the reader to \cite{kelly} for the foundations of enriched category
theory and to \cite{toen:morita} for more details on derived Morita theory. Given dg categories $\A$, $\B$, there is a dg category
\[
	\A \otimes \B 
\]
called the {\em tensor product} of $\A$ and $\B$. A {\em dg functor}  $\A \to \B$ of dg categories is
defined to be a $\Ch(k)$-enriched functor. The collection of functors from $\A $ to $\B$ organize
into a dg category
\[
	\UFun(\A,\B)
\]
which is adjoint to the above tensor product. 

Given a dg category $\A$, we introduce the dg category
\[
	\Mod_{\A} = \UFun(\A^{\op}, \Ch(k))
\]
of {\em right modules over $\A$}. We will mostly use right modules, but it is notationally convenient to further
introduce the dg category
\[
	\Mod^{\A} = \UFun(\A, \Ch(k))
\]
of {\em left modules over $\A$}, and, given another dg category $\B$, the dg category
\[
	\Mod^{\A}_\B = \UFun(\A \otimes \B^{\op}, \Ch(k)) \cong \UFun(\A, \Mod_\B)
\]
of {\em $\A$-$\B$-bimodules}. 
There is a canonical dg functor
\[
	\A \lra \Mod_{\A},\; a \mapsto \A(-,a) 
\]
given by the $\Ch(k)$-enriched Yoneda embedding.
Letting $\Ae = \A^{\op} \otimes \A$, the corresponding bimodule
	\[
		\A: \A^{\op} \otimes \A \lra \Mod_k,\; (a,a') \mapsto \A(a,a').
	\]
is called the {\em diagonal bimodule}.	

The category $\Mod_{\A}$ admits a natural cofibrantly generated $\Ch(k)$-model structure in the
sense of \cite{hovey}: it is obtained from the projective model structure on $\Ch(k)$ by defining
weak equivalences and fibrations pointwise. 

\paragraph{Hochschild homology.} Let $\A$ be a dg category. We define the {\em Hochschild complex}
\[
	\CB{\A} = \A \otimes^L_{\Ae} \A
\]
where $\A$ denotes the diagonal $\A$-bimodule. If we use the bar resolution of $\A$ as a particular
choice of cofibrant replacement, then the right-hand side complex becomes the cyclic bar construction.
This complex arises as the realization of a cyclic object in $\Ch(k)$ which equips $\CB{\A}$ with an
action of the circle $S^1$ (cf. \cite{hoyois:cyclic}).  This model further exhibits an explicit
functoriality: a functor $f: \A \to \B$ induces an $S^1$-equivariant morphism
\[
	\CB{\A} \lra \CB{\B}.
\]
In virtue of the circle action, we obtain the {\em cyclic complex}
\[
	\CC{\A} = (\CB{\A}[u^{-1}], b + u B)
\]
by passing to homotopy orbits and the {\em negative cyclic complex}
\[
	\NC{\A} = (\CB{\A}[ [u]], b + u B)
\]
by passing to homotopy fixed points. As explained in \cite{hoyois:cyclic}, the circle action
is captured algebraically in terms of the structure of a mixed complex so that the above orbit and fixed point
constructions can be computed by the well-known complexes (cf. \cite{kassel:cyclic, loday:cyclic}).

Given a functor $f: \A \to \B$ of dg categories, we define the
{\em relative Hochschild complex } $\CB{\B,\A}$ as the cofiber (or cone) of the morphism $\CB{\A} \to
\CB{\B}$. Similarly, we obtain the relative cyclic complex $\CC{\B,\A}$ and the relative negative cyclic
complex $\NC{\B,\A}$.

\paragraph{Derived $\infty$-categories.} It will be convenient to formulate some of the constructions below in
terms of $\infty$-categories. Given a dg category $\A$, let $\Mod_{\A}^{\circ}$ denote the full dg
subcategory of $\Mod_{\A}$ spanned by the cofibrant objects. We call the dg nerve
\[
	\D(\Mod_{\A}) = \Ndg(\Mod_{\A}^{\circ})
\]
the derived $\infty$-category of $\A$-modules (cf. \cite{lurie:algebra}).

\paragraph{Morita localization.} A dg functor $f: \A \to \B$ is called a {\em quasi-equivalence} if
the following hold:
\begin{enumerate}
	\item the functor $H^0(\A) \to H^0(\B)$ is an equivalence of categories,
	\item for every pair $(a,a')$ of objects in $\A$, the map
		\[
			\A(a,a') \to \B(f(a),f(a'))
		\]
		is a quasi-isomorphism of complexes.
\end{enumerate}
Given a dg category $\A$, we define the dg category $\Perf_\A$ of 
{\em perfect $\A$-modules} as the full dg subcategory of $\Mod_{\A}$ spanned by those cofibrant
objects which are compact in $\Ho(\Mod_\A)$. Here, an object $M$ in $\Mod_\A$ is called {\em
compact}, if the functor
\[
	\UHom_\A(M, -):\; \Mod_\A \lra \Mod_k
\]
commutes with filtered homotopy colimits. In fact, to show that $M \in \Mod_\A$ is compact it is enough to check that the functor $\UHom_\A(M, -)$ preserves abitrary direct sums. Furthermore, it is known that the compact objects in $\Mod_\A$ are precisely the homotopy retracts of finite colimits of representable modules. A dg functor $f:\A \to \B$ induces via enriched 
left Kan extension a functor
\[
	f_!:\; \Perf_{\A} \lra \Perf_{\B}.
\]
We say $f$ is a {\em Morita equivalence} if the functor $f_!$ is a quasi-equivalence. We are usually
interested in dg categories up to quasi-equivalence (resp. Morita equivalence) and will implement
this by working with the $\infty$-categories $\Lqe$ (resp. $\Lmo$) obtained by
localizing $\Catdg(k)$ along the respective collection of morphisms.

\subsection{Morita theory}

Let $\A$, $\B$ dg categories and let $M \in \Mod^\A_\B$ be a cofibrant bimodule. We denote by
\[
	- \otimes_\A M: \Mod_\A \lra \Mod_\B
\]
the $\Ch(k)$-enriched left Kan extension of $M: \A \to \Mod_\B$ along the Yoneda embedding $\A \to \Mod_\A$. We
further introduce the dg functor $\UHom_\B(M,-): \Mod_\B \to \Mod_\A$ given by the composite 
\[
	\Mod_\B \overset{M^{\op} \otimes \one}{\lra} \UFun(\A^{\op}, \Mod_\B^{\op} \otimes \Mod_\B)
	\overset{\UHom_\B(-,-)}{\lra} \Mod_\A,
\]
obtaining a $\Ch(k)$-enriched Quillen adjunction
\[
	- \otimes_\A M: \Mod_\A \llra \Mod_\B: \UHom_\B(M,-).
\]
Concretely, the dg functor $\UHom_\B(M,-): \Mod_\B \to \Mod_\A$ takes a module $N \in \Mod_\B$ to the $\A$-module $a \mapsto \Hom_\B(M(a),N)$.
We call the bimodule $M^{\vee}: \B \to \Mod_\A$ given by the restriction of $\UHom_\B(M,-)$ along the enriched Yoneda
embedding $\B \lra \Mod_\B$ the {\em right dual of $M$}. By the universal property of its enriched left
Kan extension
\[
	- \otimes_\B M^{\vee}: \Mod_\B \lra \Mod_\A,
\]
we obtain a canonical natural transformation
\[
	\eta: - \otimes_\B M^{\vee} \lra \UHom_\B(M,-).
\]

\begin{defi} The bimodule $M$ is called {\em right dualizable} if, for every cofibrant $\B$-module $N$,
	the morphism $\eta(N)$ in $\Mod_\A$ is a weak equivalence.
\end{defi}

\begin{rem} A bimodule $M \in \Mod^\A_\B$ is right dualizable if and only if, for every $a \in \A$, the
	right $\B$-module $M(a)$ is perfect. 
\end{rem}

\begin{rem}\label{rem:right-unit} A right dualizable bimodule $M \in \Mod^\A_\B$ induces a
	$\D(\Ch(k))$-enriched adjunction of derived categories
	\[
		- \otimes^L_\A M: \D(\Mod_\A) \llra \D(\Mod_\B): - \otimes^L_\B M^{\vee}
	\]
	with unit and counit induced via Kan extensions from bimodule morphisms
	\[
		\A \lra M \otimes_\B^L M^{\vee}
	\]
	and
	\[
		M^{\vee} \otimes_\A^L M \lra \B,
	\]
	respectively. 
		
\end{rem}

Dually, we may consider $M \in \Mod^\A_\B$ as a dg functor
\[
	M: \B^{\op} \lra \Mod^\A.
\]
By the above construction, we obtain a $\Ch(k)$-enriched Quillen adjunction
\[
	M \otimes_\B - :\Mod^\B \llra \Mod^\A : \UHom^\A(M,-)
\]
and call the bimodule $\presuper{^\vee}{M} \in \Mod^\B_\A$ given by restricting $\UHom^\A(M,-)$ along $\A^{\op} \to
\Mod^\A$ the {\em left dual of $M$}. From the universal property of left Kan extension, we obtain a canonical natural transformation
\begin{equation}\label{eq:xi}
		\xi: \presuper{^\vee}{M} \otimes_\A - \lra \UHom^\A(M,-).
\end{equation}

\begin{defi} The bimodule $M \in \Mod^\A_\B$ is called {\em left dualizable} if, for every cofibrant
	$N \in \Mod^\A$, the morphism $\xi(N)$ is a weak equivalence in $\Mod^\B$.
\end{defi}

\begin{rem} A bimodule $M \in \Mod^\A_\B$ is left dualizable if and only if, for every $b \in \B$, the
	left $\A$-module $M(b)$ is perfect. 
\end{rem}

\begin{rem}\label{rem:left-unit} A left dualizable bimodule $M \in \Mod^\A_\B$ induces 
	a $\D(\Ch(k))$-enriched adjunction of derived categories
	\[
		M \otimes_\B^L -: \D(\Mod^\B) \llra \D(\Mod^\A): \presuper{^{\vee}}{M} \otimes^L_\A - 
	\]
	with unit and counit induced via Kan extensions from bimodule morphisms
	\[
		\B \lra  \presuper{^{\vee}}{M}  \otimes^L_\A M 
	\]
	and
	\[
		M \otimes^L_\B \presuper{^{\vee}}{M} \lra \A,
	\]
	respectively. 
\end{rem}

\begin{prop}\label{prop:morita-adj}
	Let $M \in \Mod^\A_\B$ be a cofibrant bimodule. 
	\begin{enumerate}
		\item Assume that $M$ is right dualizable. Then
	the cofibrant replacement $Q(M^{\vee}) \in \Mod^\B_\A$ of the right dual of $M$ is left dualizable and its
	left dual is canonically equivalent to $M$. 
		\item Assume that $M$ is left dualizable. Then
			the cofibrant replacement $Q(\presuper{^{\vee}}{M}) \in \Mod^\B_\A$ of the
			left dual of $M$ is right dualizable and its right dual is canonically equivalent to $M$. 
	\end{enumerate}
\end{prop} 
\begin{proof}
	We give an argument for statement (1). We observe that the unit and counit morphisms for the
	adjunction
	\[
		- \otimes^L_\A M: \D(\Mod_\A) \llra \D(\Mod_\B): \RHom_\B(M,-)
	\]
	from Remark \ref{rem:right-unit} can be interpreted as unit and counit morphisms for the
	adjunction
	\[
		M^{\vee} \otimes_\A^L -: \D(\Mod^\A) \llra \D(\Mod^\B): \RHom^\B(M^{\vee},-)
	\]
	from Remark \ref{rem:left-unit}. The statements now follow from the uniqueness of right
	adjoints.
\end{proof}

\begin{cor}\label{cor:morita-adj} Let $M \in \Mod^\A_\B$ be a cofibrant right dualizable bimodule. 
	Then there are $\D(\Ch(k))$-enriched adjunctions
	\[
		- \otimes^L_\A M: \D(\Mod_\A) \llra \D(\Mod_\B): - \otimes^L_\B M^{\vee}
	\]
	and
	\[
		M^{\vee} \otimes^L_\A - : \D(\Mod^\A) \llra \D(\Mod^\B): M \otimes^L_\B - .
	\]
	Similarly, let $M \in \Mod^\A_\B$ be a left dualizable bimodule. 
	Then there are $\D(\Ch(k))$-enriched adjunctions
	\[
		M \otimes^L_\B -: \D(\Mod^\B) \llra \D(\Mod^\A): \presuper{^{\vee}}{M} \otimes^L_\A -
	\]
	and
	\[
		- \otimes^L_\B \presuper{^{\vee}}{M} : \D(\Mod_\B) \llra \D(\Mod_\A): - \otimes^L_\A M .
	\]
\end{cor}

\begin{exa} 
	The dg category $\A$ is called {\em (locally) proper} if the diagonal bimodule, considered as an object
	of $\Mod^{\Ae}_k$, is right dualizable. Concretely, $\A$ is proper if $\A(a,a^{\prime}) \in \Perf_{k}$ for all $(a,a^{\prime}) \in \A^{e}$. We denote the right dual of $\A$ by $\A^*$. 
	Assuming $\A$ is proper, we have, by Corollary \ref{cor:morita-adj}, adjunctions
	\[
			- \otimes^L_{\Ae} \A: \D(\Mod_{\Ae}) \llra \D(\Mod_k): - \otimes^L_k \A^*
	\]
	and 
	\[
		\A^* \otimes^L_{\Ae} -: \D(\Mod^{\Ae}) \llra \D(\Mod^k): \A \otimes^L_k -.
	\]
	In particular, we obtain canonical equivalences in $\D(\Ch(k))$
\begin{equation}\label{eq:proper-HH}
		\RHom_k(\A \otimes^L_{\Ae} \A, k) \overset{\simeq}{\lra} \RHom_{\Ae}(\A, \A^*) 
\end{equation}
and
	\[
			\RHom_k(\A^* \otimes^L_{\Ae} \A, k) \overset{\simeq}{\lra} \RHom_{\Ae}(\A, \A) 
	\]
giving descriptions of the $k$-linear dual of Hochschild homology and of Hochschild cohomology,
respectively, in terms of $\A^*$. In this context, we will also refer to $\A^*$ as the {\em dualizing bimodule}. 
\end{exa}

\begin{exa} Let $\A$ be a dg category. The dg category $\A$ is called {\em smooth} if the diagonal bimodule, considered as an object
	$\A \in \Mod^{\Ae}_k$, is left dualizable. Concretely, $\A$ is smooth if the diagonal
	bimodule $\A$ is perfect as a bimodule. We denote the left dual of $\A$ by $\A^!$. 
	Assuming $\A$ is smooth, we have, by Corollary \ref{cor:morita-adj}, adjunctions
	\[
		\A \otimes^L_{k} -: \D(\Mod^{k}) \llra \D(\Mod^{\Ae}): \A^! \otimes^L_{\Ae} -
	\]
	and 
	\[
		- \otimes^L_{k} \A^! : \D(\Mod_{k}) \llra \D(\Mod_{\Ae}): - \otimes^L_{\Ae} \A.
	\]
	In particular, we obtain canonical equivalences in $\D(\Ch(k))$
	\[
		\RHom_{\Ae}(\A, \A) \overset{\simeq}{\lra} \A^! \otimes^L_{\Ae} \A
	\]
	and
	\begin{equation}\label{eq:smooth-HH}
		\RHom_{\Ae}(\A^!, \A) \overset{\simeq}{\lra} \A \otimes^L_{\Ae} \A
	\end{equation}
	providing descriptions of Hochschild cohomology and homology, respectively, in terms of
	$\A^!$. We will refer to $\A^!$ as the {\em inverse dualizing bimodule}.
\end{exa}

\subsection{Duality}

Let $\A$ be a smooth dg category. An object $p \in \A$ is called {\em locally perfect} if, for every
$a \in \A$, the mapping complex $\A(a,p)$ is perfect. Let $\P \subset \A$ be a full dg subcategory of
$\A$ spanned by some collection of locally perfect objects. Consider the bimodule
\[
	\AP: \A^{\op} \otimes \P \lra \Mod_k, \; (a,p) \mapsto \A(a,p)
\]
as an object of $\Mod^\P_\A$. Since $\P$ consists of locally perfect objects, the bimodule $\AP$ is
right dualizable and we denote its right dual by $\AP^* \in \Mod^\A_\P$ so that we have an adjunction
\[
	- \otimes^L_{\A^{\op} \otimes \P} \AP : \D(\Mod_{\A^{\op} \otimes \P}) \llra \D(\Mod_k): -
	\otimes^L_k \AP^*.
\]
On the other hand, since $\A$ is smooth, we have an adjunction 
\[
	- \otimes^L_{k} \A^! : \D(\Mod_{k}) \llra \D(\Mod_{\Ae}): - \otimes^L_{\Ae} \A
\]
where $\A^! \in \Mod^\A_\A$ is the inverse dualizing bimodule. We denote by $\PA^!$ the restriction of
$\A^!: \A \to \Mod_\A$ along $\P \subset \A$.

\begin{prop} Let $\A$ be a smooth dg category.
	\begin{enumerate} 
			\item The above bimodules $\AP^*$ and $\PA^!$ form an adjunction
			\[
				- \otimes^L_P \PA^!: \D(\Mod_\P) \llra \D(\Mod_\A): - \otimes^L_\A \AP^*.
			\]
			\item Assume in addition that $\A$ is proper. Then we may set $\P=\A$ and the
				adjunction becomes an equivalence
			\[
				- \otimes^L_\A \A^!: \D(\Mod_\A) \overset{\simeq}{\llra} \D(\Mod_\A): -
				\otimes^L_\A \A^*.
			\]
		\end{enumerate}
\end{prop}
\begin{proof}
	To show (1), we may describe $- \otimes^L_\P \PA^!$ as a Morita composite of the functors 
	\[
		\P \otimes_k \A^! \in \Mod^\P_{\P \otimes \A^{\op} \otimes \A}
	\]
	and 
	\[
		\PA \otimes_k \A \in \Mod^{\P \otimes \A^{\op} \otimes \A}_\A
	\]
	so that 
	\[
		- \otimes^L_\P \PA^! \simeq - \otimes_\P^L (\P \otimes_k \A^!) \otimes^L_{\P \otimes
			\A^{\op} \otimes \A} (\PA \otimes_k \A).
	\]
	Passing to right adjoints, we obtain that 
	\begin{align*}
		(- \otimes^L_\P \PA^!)^R & \simeq \otimes^L_\A (\PA \otimes_k \A)^R \otimes^L_{\P \otimes
			\A^{\op} \otimes \A} (\P \otimes_k \A^!)^R\\
					& \simeq \otimes^L_\A (\AP^* \otimes_k \A) \otimes^L_{\P \otimes
			\A^{\op} \otimes \A} (\P \otimes_k \A)\\
					& \simeq \otimes^L_\A \AP^*
	\end{align*}
	proving the claim. Statement (2) follows from a similar calculation. 
\end{proof}

\begin{exa} Let $\A$ be a smooth dg category, let $a \in \A$ be any object, and let $p \in \A$ be a locally perfect object. Then 
	we have an equivalence
	\begin{equation}\label{eq:duality}
			\RHom_\A(\A^!(-,p), \A(-,a)) \simeq \RHom_\P(\P(-,p), \AP^*(-,a)) \simeq \RHom_k(\A(a,p),k)
	\end{equation}
	in $\D(\Ch(k))$. Assume, in addition, that $\A$ is proper. Then, restricting to compact
	objects, we obtain inverse autoequivalences
	\[
		- \otimes^L_\A \A^!: \Perf_\A \llra \Perf_\A : - \otimes^L_\A \A^*
	\]
	so that, for any pair of perfect modules $M,N$, we have 
	\[
		\RHom_\A(M, N \otimes_\A \A^*) \simeq \RHom_k(\RHom_\A(N, M), k).
	\]
	In this situation, the autoequivalence $\A^*$ is known as the Serre functor so that $\A^!$
	becomes the inverse of the Serre functor. In a suitable geometric contex, this functor can
	be described in terms of a dualizing complex which explains the terminology {\em (inverse)
	dualizing bimodule}.
\end{exa}

\section{Absolute Calabi-Yau structures} 
\label{sec:absolute}

\subsection{Right Calabi-Yau structures}

Let $\A$ be a proper dg category so that, by \eqref{eq:proper-HH}, we have an equivalence
\begin{equation}\label{eq:Psi}
	\Psi: \RHom_k(\A \otimes^L_{\Ae} \A, k) \overset{\simeq}{\lra} \RHom_{\Ae}(\A, \A^*).
\end{equation}

\begin{defi} An $n$-dimensional {\em right Calabi-Yau structure} on $\A$ consists of a map of complexes
	\[
		\widetilde{\omega}: \CC{\A} \lra k[-n]
	\]
	such that the corresponding morphism of $\A$-bimodules
	\[
		\Psi(\omega): \A[n] \lra \A^*
	\]
	is a weak equivalence. Here $\omega$ denotes the pullback of $\widetilde{\omega}$ along
	$\CB{\A} \to \CC{\A}$. 
\end{defi}

\begin{rem} A right Calabi-Yau structure identifies the
	diagonal bimodule $\A$ up to shift with its right dual $\A^*$.
\end{rem}

\subsection{Left Calabi-Yau structures}

Let $\A$ be a smooth dg category so that, by \eqref{eq:smooth-HH}, we have an equivalence
\begin{equation}\label{eq:Phi}
		\Phi:  \A \otimes^L_{\Ae} \A \overset{\simeq}{\lra} \RHom_{\Ae}(\A^!, \A) .
\end{equation}

\begin{defi} Let $\A$ be a smooth dg category. An $n$-dimensional {\em left Calabi-Yau structure} on $\A$
	consists of a map of complexes
	\[
		\widetilde{[\A]}: k[n] \lra \NC{\A}
	\]
	such that the corresponding morphism of $\A$-bimodules
	\[
		\Phi([\A]): \A^! \lra \A[-n]
	\]
	is a weak equivalence. Here $[\A]$ denotes the postcomposition of $\widetilde{[\A]}$ with
	$\NC{\A} \to \CB{\A}$.
\end{defi}

\begin{rem} A left Calabi-Yau structure identifies the
	diagonal bimodule $\A$ up to shift with its left dual $\A^!$.
\end{rem}

Given a dg category $\A$ and a full dg subcategory $\P \subset \A$ spanned by locally perfect objects,
we obtain a dg functor
\[
	\A^{\op} \otimes \P \lra \Perf_k, \; (a,p) \mapsto \A(a,p).
\]
Applying $\CB{-}$, we obtain an adjoint morphism of complexes
\[
	\CB{\A} \simeq \CB{\A^{\op}} \lra \RHom(\CB{\P}, \CB{\Perf_k}) \simeq \RHom_k(\CB{\P}, k)
\]
which is compatible with the circle actions (it can be realized as a map of cyclic complexes), so
that upon passing to homotopy fixed points, we obtain a map
\[
	\widetilde{\Theta}: \NC{\A} \lra \RHom_k(\CC{\P},k)
\]
which is part of a commutative square
\begin{equation}\label{eq:lift-square}
		\xymatrix{\NC{\A} \ar[r]^-{\widetilde{\Theta}} \ar[d] & \RHom_k(\CC{\P},k)\ar[d]\\
		\CB{\A} \ar[r]^-{\Theta} & \RHom_k(\CB{\P},k).}
\end{equation}

\begin{thm}\label{theo:LCY-RCY} Let $\A$ be a smooth dg category equipped with an $n$-dimensional left Calabi-Yau
	structure $\widetilde{[\A]}$. Let $\P \subset \A$ be a full dg subcategory spanned by a set of
	locally perfect objects. Then the map $\widetilde{\Theta}(\widetilde{[\A]})$ provides an
	$n$-dimensional right Calabi-Yau structure on $\P$.
\end{thm}
\begin{proof}
	Using the equivalences \eqref{eq:Phi} and \eqref{eq:Psi}, we may augment 
	\eqref{eq:lift-square} by the square 
	\[
		\xymatrix{\CB{\A} \ar[r]^-{\Theta}\ar[d]^{\Phi}_{\simeq} &
		\RHom_k(\CB{\P},k)\ar[d]^{\Psi}_{\simeq}\\
		\RHom_{\Ae}(\A^!,\A) \ar[r]^-{\Theta'} & \RHom_{\Pe}(\P,\P^*).}
	\]
	The map $\Theta'$ admits the following description: Consider the dg functor $D$ given by the
	composite
	\[
		\Mod^\A_\A \lra \Mod^\P_\A \overset{M \mapsto M^{\vee}}{\lra} (\Mod^\A_\P)^{\op} \lra
		(\Mod^\P_\P)^{\op}
	\]
	where the first and last functors are given by restriction along $\P \subset \A$. We obtain an
	induced $\D(\Ch(k))$-enriched functor
	\[
		LD: \D(\Mod_{\Ae}) \lra \D(\Mod_{\Pe})^{\op}.
	\]
	Explicitly, this functor associates to an $\A$-bimodule $M$, the $\P$-bimodule given by
	\[
		(p,p') \mapsto \RHom_{\A}(M(-,p'), \A(-,p)).
	\]
	Therefore, the functor $LD$ maps the diagonal bimodule $\A$ to the diagonal bimodule
	$\P$, and, by \eqref{eq:duality}, the inverse dualizing bimodule $\A^!$ to the dualizing
	bimodule $\P^*$. An explicit calculation shows that the map $\Theta'$ is the map
	induced by $LD$ on mapping complexes. In particular, $\Theta'$ preserves equivalences: the equivalence
	\[
		\Phi([\A]): \A^! \overset{\simeq}{\lra} \A[-n]
	\]
	maps to an equivalence
	\[
		\Theta'(\Phi([\A])): \P \overset{\simeq}{\lra} \P^*[-n]
	\]
	showing that $\widetilde{\Theta}(\widetilde{[\A]})$ is indeed a right Calabi-Yau structure.
\end{proof}

\begin{rem} Essentially all examples of right Calabi-Yau structures in this work are induced from
	left Calabi-Yau structures via the construction of Theorem \ref{theo:LCY-RCY}. Therefore, it
	seems that smooth (or even finite type) dg categories equipped with left Calabi-Yau
	structures should be considered as the fundamental objects.
\end{rem}

\section{Relative Calabi-Yau structures}
\label{sec:relative}

\subsection{Relative right Calabi-Yau structures}

Let $f: \A \to \B$ be a dg functor of proper dg categories. We have an induced morphism
\[
	\CC{f}: \CC{\A} \lra \CC{\B} 
\]
defined explicitly in terms of cyclic bar constructions. We abbreviate $\RHom_k(-,k)$ by $(-)^*$. Using \eqref{eq:smooth-HH}, we obtain a
coherent diagram 
\begin{equation}\label{eq:diag-proper}
		\xymatrix{ \CC{\A}^* \ar[r]& \CB{\A}^* \ar[r]^-{\Psi_\A}_-{\simeq} & \RHom_{\Ae}(\A, \A^*) \\
		\CC{\B}^* \ar[r]\ar[u]^{\CC{f}^*}  & \CB{\B}^* \ar[u]_{\CB{f}^*}\ar[r]^-{\Psi_\B}_-{\simeq} &
		\RHom_{\Be}(\B, \B^*) \ar[u]^{\Psi_f}}
\end{equation}
in $\D(\Ch(k))$. We give an explicit description of the map $\Psi_f$:
Consider the $\Ch(k)$-enriched Quillen adjunction
\begin{equation}\label{eq:F-adjunction}
		F_!: \Mod_{\Ae} \llra \Mod_{\Be} : F^*
\end{equation}
with $F = f^{\op} \otimes f$. 
We introduce the morphism 
\begin{equation}\label{eq:F-unit}
		u: \A \lra F^*\B
\end{equation}
in $\D(\Mod_{\Ae})$ and its derived left adjoint
\begin{equation}\label{eq:F-counit}
		c: LF_!\A \lra \B
\end{equation}
in $\D(\Mod_{\Be})$.

\begin{rem} The morphisms $u$ and $c$ represent unit and counit, respectively, of the derived
	adjunction
	\[
		Lf_!: \D(\Mod_\A) \llra \D(\Mod_\B): Rf^*.
	\]
\end{rem}

\begin{prop}\label{prop:Psi_f} The morphism $\Psi_f$ is given by the composite 
\[
	\RHom_{\Be}(\B, \B^*) \overset{F^*}{\lra} \RHom_{\Ae}(F^*\B, (F^*\B)^*) \overset{u}{\lra} \RHom_{\Ae}(\A, \A^*)
\]
where we implicitly use the canonical identification $F^*(\B^*) \simeq (F^*\B)^*$.
\end{prop}
\begin{proof} Follows from an explicit calculation using the bar resolution.
\end{proof}

Let $\widetilde{\omega}: \CC{\B,\A} \to k[-n]$ be a morphism of complexes. We interpret $\omega$ as a coherent
diagram
\[
	\xymatrix{ k[-n+1] \ar[r]\ar[d] & \ar[d]^{\CC{f}^*} \CC{\B}^*\\
	0 \ar[r] & \CC{\A}^* }
\]
in $\D(\Ch(k))$. By forming the composite with \eqref{eq:diag-proper}, we
obtain the coherent diagram
\[
	\xymatrix{ k[-n+1] \ar[r]\ar[d] & \RHom_{\Be}(\B, \B^*) \ar[d]^{\Psi_f}\\
	0 \ar[r] & \RHom_{\Ae}(\A, \A^*) }
\]
from which we extract the datum of a morphism $\xi: \B[n-1] \to \B^*$ together with a chosen zero homotopy of the
morphism $\Psi_f(\xi): \A[n-1] \to \A^*$. By Proposition \ref{prop:Psi_f}, this morphism can be
identified with the composite
\[
	\A[n-1] \overset{u}{\lra} F^*\B[n-1] \overset{F^*\xi}{\lra} (F^*\B)^*
	\overset{c}{\lra} \A^*
\]
so that the chosen zero homotopy induces the dashed arrows which make the diagram
\begin{equation}\label{eq:proper-nondegenerate}
	\xymatrix@C=10ex{
		\A[n-1] \ar[r]^{u} \ar@{-->}[d]^{\xi'} & F^*\B[n-1] \ar[d]^{\xi} \ar[r] & \cof(u) 
		\ar@{-->}[d]^{\xi''}\\
		\fib(u^*) \ar[r] &  (F^*\B)^* \ar[r]^{u^*} & \A^*
	}
\end{equation}
in $\D(\Mod_{\Ae})$ coherent.
		
\begin{defi}
	An $n$-dimensional {\em right Calabi-Yau structure} on $f$ consists of a morphism
	\[
		\widetilde{\omega_{(\B,\A)}}: \CC{\B,\A} \lra k[-n]
	\]
	such that all vertical morphisms in the corresponding diagram \eqref{eq:proper-nondegenerate} are
	equivalences in $\D(\Mod_{\Ae})$.
\end{defi}

\begin{exa} Let $\A$ be a proper dg category and consider the zero functor $f: \A \to 0$. Then an
	$n$-dimensional right Calabi-Yau structure on $f$ translates to a morphism
	\[
		\widetilde{\omega_{\A}}: \CC{\A} \lra k[-n]	
	\]
	such that the vertical maps in 
	\[
	\xymatrix@C=10ex{
		\A[-n+1] \ar[r] \ar@{-->}[d]^{\xi'} & 0 \ar[d] \ar[r] & \A[-n] 
		\ar@{-->}[d]^{\xi''}\\
		\A^*[1] \ar[r] & 0 \ar[r] & \A^*.
	}
	\]
	are equivalences. But this datum is equivalent to an absolute right Calabi-Yau structure on
	$\A$. 
\end{exa}

\subsection{Relative left Calabi-Yau structures}

Let $f: \A \to \B$ be a functor of smooth dg categories. We have an induced morphism
\[
	\NC{f}: \NC{\A} \lra \NC{\B} 
\]
defined explicitly in terms of cyclic bar constructions. Using \eqref{eq:smooth-HH}, we obtain a
coherent diagram 
\begin{equation}\label{eq:diag-smooth}
		\xymatrix{ \NC{\A} \ar[r]\ar[d]^{\NC{f}} & \CB{\A}\ar[d]^{\CB{f}} \ar[r]^-{\Phi_\A}_-{\simeq} & \RHom_{\Ae}(\A^!, \A) \ar[d]^{\Phi_f}\\
		\NC{\B} \ar[r] & \CB{\B} \ar[r]^-{\Phi_\B}_-{\simeq} & \RHom_{\Be}(\B^!, \B) }
\end{equation}
in $\D(\Ch(k))$. Just like for proper dg categories, we give an explicit description of the map
$\Phi_f$ which will now involve the {\em counit morphism}
\[
	c: LF_!\A \lra \B
\]
in $\D(\Mod_{\Be})$ from \eqref{eq:F-counit}.

\begin{lem}\label{lem:delta} Let $M \in \Perf_{\Ae}$ be a perfect bimodule. Then there is a canonical equivalence
	\[
		\delta: LF_!(M^!) \simeq (LF_! M)^!
	\]
	in $\D(\Mod_{\Be})$.
\end{lem}
\begin{proof} We may interpret $M$ as a is left dualizable module $M \in \Mod^{\Ae}_k$. Therefore,
	using \eqref{eq:xi}, we have equivalences
	\[
		LF_!(M^!) \simeq M^! \otimes_{\Ae}^L F^*\Be \overset{\xi}{\simeq} \RHom_{\Ae}(M,F^*\Be)
		\simeq \RHom_{\Be}(LF_!M,\Be) \simeq (LF_!M)^!.
	\]
\end{proof}

\begin{prop}\label{prop:Phi_f} The morphism $\Phi_f$ is given by the composite 
\[
	\RHom_{\Ae}(\A^!, \A) \overset{LF_!}{\lra} \RHom_{\Ae}((LF_!\A)^!, LF_!\A) \overset{c}{\lra} \RHom_{\Be}(\B^!, \B).
\]
where we implicitly use the identification $\delta: LF_!(\A^!) \simeq (LF_!\A)^!$ from Lemma \ref{lem:delta}.
\end{prop}
\begin{proof}
	Follows from an explicit calculation using the bar resolution.
\end{proof}

Let $\sigma: k[n] \to \NC{\B,\A}$ be a negative cyclic cycle. We may interpret $\sigma$ as a coherent
diagram
\[
	\xymatrix{ k[n-1] \ar[r]\ar[d] & \ar[d]^{\NC{f}} \NC{\A}\\
	0 \ar[r] & \NC{\B} }
\]
in $\D(\Ch(k))$. By forming the composite with \eqref{eq:diag-smooth}, we
obtain the coherent diagram
\[
	\xymatrix{ k[n-1] \ar[r]\ar[d] & \RHom_{\Ae}(\A^!, \A) \ar[d]^{\Phi_f}\\
	0 \ar[r] & \RHom_{\Be}(\B^!, \B) }
\]
from which we extract the datum of a morphism $\xi: \A^! \to \A[-n+1]$ together with a chosen zero homotopy of the
morphism $\Phi_f(\xi): \B^! \to \B[-n+1]$. By Proposition \ref{prop:Phi_f}, this morphism can be
identified with the composite
\[
	\B^! \overset{c^!}{\lra} (LF_!\A)^! \simeq LF_!(\A^!) \overset{LF_!(\xi)}{\lra} LF_!(\A[-n+1])
	\overset{c}{\lra} \B[-n+1]
\]
so that the chosen zero homotopy induces the dashed arrows which make the diagram
\begin{equation}\label{eq:smooth-nondegenerate}
	\xymatrix@C=10ex{
		\B^! \ar[r]^{c^!} \ar@{-->}[d]^{\xi'} & (LF_!\A)^! \ar[d]^{\xi} \ar[r] & \cof(c^!) 
		\ar@{-->}[d]^{\xi''}\\
		\fib(c) \ar[r] & LF_!\A[-n+1] \ar[r]^{c} & \B[-n+1]
	}
\end{equation}
in $\D(\Mod_{\Be})$ coherent.
		
\begin{defi}
	An $n$-dimensional {\em left Calabi-Yau structure} on $f$ consists of a morphism
	\[
		\widetilde{[\B,\A]}: k[n] \lra \NC{\B,\A}
	\]
	such that all vertical morphisms in the corresponding diagram \eqref{eq:smooth-nondegenerate} are
	equivalences in $\D(\Mod_{\Be})$.
\end{defi}

\begin{exa} Let $\B$ be a smooth dg category. Consider the zero functor $f: 0 \to \B$ (up to Morita
	equivalence we may assume that $\B$ has a zero object). Then an $n$-dimensional relative left Calabi-Yau
	structure on $f$ translates to a morphism
	\[
		\widetilde{[\B]}:k[n] \lra \NC{\B}	
	\]
	such that the vertical maps in 
	\[
	\xymatrix@C=10ex{
		\B^! \ar[r] \ar@{-->}[d]^{\xi'} & 0 \ar[d] \ar[r] & \B^![1] 
		\ar@{-->}[d]^{\xi''}\\
		\B[-n] \ar[r] & 0 \ar[r] & \B[-n+1].
	}
	\]
	are equivalences. But this datum is equivalent to an absolute left Calabi-Yau structure on
	$\B$. 
\end{exa}

\section{Examples}
\label{sec:examples}

\subsection{Topology}

The material in this section is inspired by various parts of \cite{lurie:ltheory}.

\subsubsection{Poincar\'e complexes}

We say a topological space $Y$ is of {\em finite type} if it is homotopy equivalent to a homotopy retract of a finite CW
complex. Let $Y$ be of finite type. We define the functor 
\begin{equation}\label{eq:adj}
	\dg: \Cat_{\infty} \overset{\fC}{\lra} \Cat_{\Delta} \overset{\N}{\lra} \Cat_{\dg}(k)
\end{equation}
given as the composite of the left Quillen functor $\fC$ from \cite{lurie:htt} and the functor $\N$
of applying normalized Moore chains to the mapping spaces. 
We set
\[
	\L(Y) := \dg(\Sing(Y))
\]

\begin{rem}
The objects of ${\L(Y)}$ can be identified with the points of
$Y$, and the mapping complex between objects $y,y' \in Y$ is quasi-isomorphic to the complex
$C_{\bullet}(P_{y,y'}Y)$ of singular chains on the space $P_{y,y'}Y$ of paths in $Y$ from $y$ to $y'$.
In particular, if $Y$ is path connected, then the choice of any point $y \in Y$ determines a
quasi-equivalence of dg categories
\[
	{\L(Y)} \simeq C_{\bullet}(\Omega_{y}Y)
\]
where the right-hand dg algebra of chains on the based loop space is interpreted as a dg category
with one object. 
\end{rem}

\begin{exa} The choice of any point $x$ on the circle $S^1$ provides a quasi-equivalence
	\[
		\dg(\Sing(S^1)) \simeq k\ZZ
	\]
	where $k\ZZ = k[t,t^{-1}]$ denotes the group algebra of the group $\ZZ = \pi_1(S^1, x)$.
\end{exa}

\begin{rem} The functor $\dg$ is weakly equivalent to the left adjoint of the Quillen adjunction
	\[
		\dg': \Cat_{\infty} \llra \Cat_{\dg}(k): \Ndg
	\]
	where $\Ndg$ denotes the differential graded nerve of \cite{lurie:algebra}.  
	From an enriched variant of the adjunction \eqref{eq:adj}, we obtain an equivalence of
	$\infty$-categories
	\[
		\D(\Mod_{{\L(Y)}}) \simeq \Fun_{\infty}(\Sing(Y),\Ndg(\Ch(k)))
	\]
	so that we are naturally led to interpret ${\L(Y)}$-modules as $\infty$-local systems of complexes of
	$k$-vector spaces on $Y$.  
\end{rem}

\begin{prop}\label{prop:dg-smooth} The dg category ${\L(Y)}$ is of finite type. In particular, it is smooth.
\end{prop}
\begin{proof} By construction, the functor $\dg$ commutes with
	homotopy colimits and takes retracts to retracts. The space $Y$ is of finite type and so
	expressable as a retract of a finite homotopy colimit of the constant diagram with value $\pt$. Finite
	type dg categories are stable under retracts and finite homotopy colimits (\cite{toen-vaquie:moduli}) so that it
	suffices to check the statement of the proposition for $Y = \pt$ where it is apparent. 
\end{proof}

The constant map $\pi: Y \to \pt$ induces a dg module 
\[
	k_Y: {\L(Y)} \overset{\dg(\pi)}{\lra} k \subset \Mod_k
\]
which corresponds to the constant local system on $Y$ with value $k$. 
We have a corresponding adjunction
\[
	- \otimes^L_{{\L(Y)}} k_Y: \D(\Mod_{{\L(Y)}}) \llra \D(\Mod_{k}): \RHom_{k}(k_Y,-)
\]
and
\[
	k_Y \otimes^L_k -: \D(\Mod^{k}) \llra \D(\Mod^{{\L(Y)}}): \RHom^{{\L(Y)}}(k_Y, -).
\]
The functors
\begin{align*}
	C_{\bullet}(Y,-) & := - \otimes^L_{{\L(Y)}} k_Y\\
	C^{\bullet}(Y,-) & := \RHom^{{\L(Y)}}(k_Y, -) 
\end{align*}
define homology, respectively cohomology, of $Y$ with $\infty$-local systems as coefficients.
We denote by
\[
	\zeta_Y \in \Mod_{{\L(Y)}}
\]
the left dual of a cofibrant replacement $Q(k_Y)$ of $k_Y \in \Mod^{\L(Y)}_k$.

\begin{prop}[Poincar\'e duality]\label{prop:poincare} The module $Q(k_Y)$ is left dualizable so that the canonical map
	\begin{equation}\label{eq:poincare}
			\zeta_Y \otimes^L_{{\L(Y)}} - \lra C^{\bullet}(Y,-)
	\end{equation}
	is an equivalence. 
\end{prop}
\begin{proof} It is immediate that $k_Y$ is locally perfect and hence right dualizable. But since, by
	Proposition \ref{prop:dg-smooth}, ${\L(Y)}$ is smooth, every cofibrant locally perfect ${\L(Y)}$-module is
	perfect \cite{toen-vaquie:moduli}. Therefore, $Q(k_Y)$ is left dualizable. 
\end{proof}

\begin{rem} Evaluating \eqref{eq:poincare} at $k_Y$, we obtain an equivalence
	\[
		C_{\bullet}(Y, \zeta_Y) \simeq C^{\bullet}(Y,k_Y) 
	\]
	which can be understood as a version of Poincar\'e duality. 
\end{rem}

\begin{rem} By definition, the functor $\zeta_Y$ is given as the restriction of the dg functor
	$C^{\bullet}(Y,-)$ along ${\L(Y)^{\op}} \to \Mod^{{\L(Y)}}$. In particular, for a point $y \in Y$,
	the complex $\zeta_Y(y)$ is the cohomology of the $\infty$-local system on $Y$ which assigns
	to a point $y'$ the complex $C_{\bullet}(P_{y,y'}Y)$ (cf. \cite{lurie:ltheory}).
\end{rem}

We define dg functors
\[
	{\L(Y)} \otimes {\L(Y)} \overset{\nabla}{\longrightarrow} \N(\fC(\Sing(Y)) \times \fC(\Sing(Y)))
	\overset{\Delta}{\longleftarrow} {\L(Y)}
\]
where the functor $\nabla$ is given by applying the Eilenberg-Zilber construction to mapping
complexes, and the functor $\Delta$ is induced by the diagonal map $Y \to Y \times Y$. We obtain a dg
functor
\[
	- \otimes_Y - : \Mod_{\L(Y)} \otimes \Mod_{\L(Y)} \overset{\Delta^*\nabla_!}{\lra} \Mod_{\L(Y)}
\]
which we call the {\em internal tensor product}. We may consider the restriction of this dg functor
to ${\L(Y)} \otimes {\L(Y)}$ as a module $M_Y \in \Mod^{\L(Y)}_{{\L(Y)}^{\op} \otimes {\L(Y)}}$ and obtain, via enriched Kan
extension, another dg functor
\[
	- \otimes_{\L(Y)} M_Y: \Mod_{\L(Y)} \lra \Mod_{{\L(Y)}^e}.
\]

\begin{rem} If the local system $\zeta_Y$ is invertible with respect to the tensor product
	$\otimes_Y$ on $\D(\Mod_{{\L(Y)}})$ then we call $\zeta_Y$ the $k$-linear {\em Spivak normal
	fibration}. In this case, the space $Y$ is therefore a topological analogue of a Gorenstein 
	variety.
\end{rem}

By Corollary \ref{cor:morita-adj} and Proposition \ref{prop:poincare}, we have an adjunction
\[
	\zeta_Y \otimes^L_{k} - : \D(\Mod^{k}) \llra \D(\Mod^{\L(Y)}): k_Y \otimes^L_{\L(Y)} - .
\]
which provides us with a canonical equivalence
\[
	\psi: C_{\bullet}(Y,k_Y) \overset{\simeq}{\lra} \RHom_{{\L(Y)}}(\zeta_Y, k_Y).
\]
Here, we slightly abuse notation and also denote the {\em right} ${\L(Y)}$-module
\[
	{\L(Y)}^{\op} \overset{\dg(\pi)^{\op}}{\lra} k^{\op} = k \subset \Mod_k
\]
by $k_Y$. 

\begin{defi} A class 
	\[
		[Y] \in H^{-n}(C_{\bullet}(Y,k_Y)) \cong H_n(Y, k)
	\]
	is called a {\em fundamental class} if the morphism
	\[
		\psi([Y]): \zeta_Y \lra k_Y[-n]
	\]
	is an equivalence of ${\L(Y)}$-modules. A pair $(Y,[Y])$ of a space equipped with a fundamental
	class is called a {\em Poincar\'e complex}. 
\end{defi}

\begin{rem} Let $(Y,[Y])$ be a Poincar\'e complex. Combining the equivalence $\psi([Y])$ with
	\eqref{eq:poincare} we obtain, for every $i$, isomorphisms
	\[
		H^{i}(Y,k) \cong H_{n-i}(Y,k)
	\]
	recovering classical Poincar\'e duality.
\end{rem}

\begin{exa} A closed topological manifold with chosen orientation provides an example of a
	Poincar\'e complex.
\end{exa}

We will now explain how a fundamental class for $Y$ gives rise to a left
Calabi-Yau structure on the dg category ${\L(Y)}$.

\begin{prop}\label{prop:jones} Let $Y$ be a topological space of finite type. Then the following hold:
	\begin{enumerate}
		\item We have canonical equivalences
			\begin{align*}
				k_Y \otimes_{\L(Y)}^L M_Y & \simeq {\L(Y)} & \zeta_Y \otimes_{\L(Y)}^L M_Y & \simeq
				{\L(Y)}^!.
			\end{align*}
		\item There is an $S^1$-equivariant equivalence
			\[
				j: \CB{{\L(Y)}} \overset{\simeq}{\lra} C_{\bullet}(LY, k)
			\]
			where $LY$ denotes the free loop space of $Y$ with circle action given
			by loop rotation. 
		\item The composite
			\[
				\alpha: C_{\bullet}(Y,k) \simeq \RHom_{\L(Y)}(\zeta_Y, k_Y)
				\overset{- \otimes^L_{\L(Y)} M_Y}{\lra} \RHom_{{\L(Y)}^e}({\L(Y)}^!, {\L(Y)})
				\overset{j}{\simeq} C_{\bullet}(LY, k)
			\]
			can be identified with the natural map induced by the inclusion of $Y$ as
			constant loops in $LY$.
	\end{enumerate}
\end{prop}
\begin{proof} Arguing for each connected component, we may assume $X = BG$ where $G$ is a
	topological group. The first statement then follows by explicit calculation. Statement (2) is shown in
	\cite{jones:cyclic} and (3) follows by a similar argument.
\end{proof}

\begin{thm}[\cite{cohen-ganatra:string}]\label{thm:loc-CY} Let $(Y,[Y])$ be a Poincar\'e complex. Then the dg category ${\L(Y)}$ is equipped with a
	canonical left Calabi-Yau structure. 
\end{thm}
\begin{proof} Since the map $\alpha$ from Proposition
	\ref{prop:jones} is $S^1$-invariant, it induces a canonical map
	\[
		\alpha: C_{\bullet}(Y,k) \lra \NC{{\L(Y)}}.
	\]
	Define $[{\L(Y)}] = \alpha([Y])$. The corresponding map ${\L(Y)}^! \to {\L(Y)}[-n]$ is an equivalence since it
	is the image of the equivalence $\zeta_Y \to k_Y[-n]$ under the dg functor $- \otimes_{\L(Y)}
	M_Y$.
\end{proof}

\subsubsection{Poincar\'e pairs}

We provide a generalization of Theorem \ref{thm:loc-CY} in the relative context. Let $X,Y$ be
topological spaces of finite type, and let $\varphi: X \to Y$ a continuous map. Applying
$\dg(\Sing(-))$ to $\varphi$, we obtain a dg functor
\[
	f: {\L(X)} \lra {\L(Y)}.
\]

\begin{lem} There is a canonical equivalence $Lf_!(\presuper{\vee}{k_X}) \simeq \presuper{\vee}{(Lf_! k_Y)}$.
\end{lem}

The equivalence $k_X \to f^* k_Y$ has an adjoint map 
\[
	\gamma: Lf_! k_X \lra k_Y
\]
with left dual
\[
	\presuper{\vee}{\gamma}: \zeta_Y \lra Lf_! \zeta_X.
\]

\begin{lem}\label{lem:comptop} The composite
	\[
		C_{\bullet}(X;k) \simeq \RHom_{\L(X)}(\zeta_X, k_X) \overset{Lf_!}{\lra}
		\RHom_{\L(Y)}(Lf_!\zeta_X, Lf_! k_X) \overset{\gamma \circ - \circ
			\presuper{\vee}{\gamma}}{\lra} \RHom_{\L(Y)}(\zeta_Y, k_Y) \simeq C_{\bullet}(Y;k)
	\]
	can be identified with the map $C_{\bullet}(\varphi;k)$.
\end{lem}

As a consequence of Lemma \ref{lem:comptop}, the choice of a class $[Y,X] \in H_n(Y,X;k)$ gives a
coherent diagram in $\D(\Mod_{\L(Y)})$ of the form
\begin{equation}\label{eq:nondeg}
	\xymatrix{
		\zeta_Y \ar[r]^-{\presuper{\vee}{\alpha}} \ar@{-->}[d] & Lf_! \zeta_X \ar[r]\ar[d] & \cof \ar@{-->}[d] \\
		\fib \ar[r] & Lf_! k_X[-n+1] \ar[r]^-{\alpha} &  k_Y. }
\end{equation}

\begin{defi} We call $[Y,X] \in H_n(Y,X;k)$ a {\em relative fundamental class} if all vertical maps
	in \eqref{eq:nondeg} are equivalences. A pair $(X \to Y, [Y,X])$ consisting of a continuous map of
	finite type topological spaces and a relative fundamental class is called a {\em Poincar\'e
	pair}.
\end{defi}

\begin{exa} A relative fundamental class for $\emptyset \to Y$ can be identified with an absolute
	fundamental class for $Y$. 
\end{exa}

\begin{exa} An example of a Poincar\'e pair is given by a compact oriented manifold $X$ with
	boundary $Y$. 
\end{exa}


\begin{thm} Let $(X \to Y, [Y,X])$ be a Poincar\'e pair. Then the corresponding dg functor
	\[
		f: {\L(X)} \lra {\L(Y)}
	\]
	of linearizations carries a canonical relative left Calabi-Yau structure. 
\end{thm}
\begin{proof} This follows from Proposition \ref{prop:jones} and Lemma \ref{lem:comptop} by applying
	the dg functor
	\[
		- \otimes_{\L(Y)} M_Y: \Mod_{\L(Y)} \lra \Mod_{{\L(Y)}^{\op} \otimes {\L(Y)}}.
	\]
\end{proof}

\subsection{Algebraic geometry}
\label{sec:algebraic}

In this subsection, we give  examples of (relative) left Calabi-Yau structures coming from
anticanonical divisors. We assume all schemes are separated and of finite type over a field $k$, which for simplicity we take to be of
characteristic zero, although the assumption on the characteristic seems to be unnecessary.

\subsubsection{Background on ind-coherent sheaves}

In this subsection, we briefly review the theory of ind-coherent sheaves, following
\cite{gaitsgory-rozenblyum}. Among other things, the theory of ind-coherent sheaves provides a functorially defined dualizing complex, which we can use to give a geometric computation of Hochschild chains for the dg category of coherent sheaves and a geometric computation for the map on Hochschild chains induced by pushforward along a proper morphism. We shall only need the `elementary' parts of the theory, which can be developed using only basic facts about ind-completion of $\infty$-categories and the notion of dualizable object in a monoidal $\infty$-category. If the reader is willing to make smoothness assumptions, then ind-coherent sheaves can be identified with (the dg derived category of) quasi-coherent sheaves, simplifying the formalism.

Given a separated scheme $X$ of finite type over a field $k$, let $\Perf(X)$ denote the dg category
of perfect complexes on $X$ and $\QCoh(X)$ the dg unbounded derived category of quasi-coherent
sheaves. It is known that $\QCoh(X)$ is compactly generated and that the compact objects in
$\QCoh(X)$ are exactly the perfect complexes $\Perf(X)$ (cf. \cite{thomason-trobaugh}), so that we have identifications
\begin{equation}
\QCoh(X) \simeq {\rm Ind}(\Perf(X)) \simeq \Mod_{\Perf(X)} 
\end{equation}
For every morphism $f : X \rightarrow Y$, we have an adjunction
\begin{equation}
f^{*} : \QCoh(Y) \leftrightarrow \QCoh(X) : f_{*}.
\end{equation}
Note that since $f^{*}$ sends perfects to perfects and hence compacts to compacts, the right adjoint $f_{*}$ preserves colimits.

Let $\Coh(X)$ denote the full dg subcategory of $\QCoh(X)$ spanned by objects with bounded coherent cohomology sheaves. By definition, ind-coherent sheaves on $X$ are obtained by ind-completion of $\Coh(X)$:
\begin{equation}
\Ind(X):={\rm Ind}(\Coh(X)) \simeq \Mod_{\Coh(X)}
\end{equation}
For smooth $X$, we have $\Coh(X)=\Perf(X)$ and so $\Ind(X)=\QCoh(X)$. For singular $X$, we have
proper inclusions $\Perf(X) \subset \Coh(X) \subset \QCoh(X)$. Ind-completion along the first and
second inclusion gives an adjunction
\begin{equation}
\Xi_X : \QCoh(X) \leftrightarrow \Ind(X) : \Psi_X
\end{equation}
in which the left adjoint is fully faithful and hence $\Psi_X$ realises $\QCoh(X)$ as a
colocalisation of $\Ind(X)$. There is moreover a natural t-structure on $\Ind(X)$ for which $\Psi_X$
is t-exact and, for every $n$, the restricted functor $\Psi_X: \Ind(X)^{\geq n} \rightarrow \QCoh(X)^{\geq n}$ is an equivalence. Taking the union over all $n$, we have $\Ind(X)^{+} \simeq \QCoh(X)^{+}$. The categories $\Ind(X)$ and $\QCoh(X)$ therefore differ only in their `tails' at $-\infty$.

Given a morphism $f : X \rightarrow Y$, we can restrict the functor $f_{*}: \QCoh(X) \rightarrow \QCoh(Y)$ to $\Coh(X)$, obtaining a functor $f_{*} : \Coh(X) \rightarrow \QCoh(X)^{+} \simeq \Ind(X)^{+} \subset \Ind(X)$. Passing to ind-completion and slightly abusing notation, we obtain a functor 
$f_{*}: \Ind(X) \rightarrow \Ind(Y)$. When $f: X \rightarrow Y$ is proper, we have an adjunction
\begin{equation}
f_{*} : \Ind(X) \leftrightarrow \Ind(Y) : f^{!}.
\end{equation}
Note that since $f$ is proper, $f_{*}$ sends compact objects to compact objects ($f_{*}\Coh(X) \subseteq \Coh(Y)$), hence the right adjoint $f^{!}$ preserves colimits. Note in particular that since we are working with separated schemes, the diagonal morphism $\Delta : X \rightarrow X \times X$ is a closed immersion, so that the functor $\Delta^{!}$ is right adjoint to $\Delta_{*}$.

Like quasi-coherent sheaves, ind-coherent sheaves enjoy good monoidal properties. In particular, there is a natural equivalence
\begin{equation}\label{boxtensor}
\boxtimes: \Ind(X) \otimes \Ind(Y) \rightarrow \Ind(X \times Y) 
\end{equation}
which, for $F \in \Coh(X)$, $G \in \Coh(Y)$, is given by the usual formula
\[
	F \boxtimes G = \pi_X^*F \otimes_{\O_X}^L \pi_Y^*G 
\]
while \eqref{boxtensor} is obtained by ind-extension (see \cite[II.2.6.3]{gaitsgory-rozenblyum}). 

Using the equivalence \ref{boxtensor}, we construct a pairing 
\begin{equation}\label{indpairing}
\xymatrix{
\Ind(X) \otimes \Ind(X) \simeq \Ind(X \times X) \ar[r]^-{\Delta^{!}} & \Ind(X) \ar[r]^-{p_{*}} & \Mod_{k}
}
\end{equation}
where $p : X \rightarrow {\rm pt}={\rm Spec}(k)$ is the structure map and we use the equivalence $\Ind({\rm pt}) \simeq \Mod_{k}$.
One can show that the pairing in \ref{indpairing} is non-degenerate in the symmetric monoidal
$\infty$-category of presentable dg categories with colimit preserving dg functors. The copairing is given by the functor 
\begin{equation}\label{indcoopairing}
\xymatrix{
\Mod_{k} \ar[r]^-{\Delta_{*}p^{!}} & \Ind(X \times X)
}
\end{equation}
where $p^{!} : \Mod_{k} \rightarrow \Ind(X)$ sends $k$ to the dualizing complex $\om^{\bullet}_{X}$
(see \cite[II.2.4.3]{gaitsgory-rozenblyum}, which also uses this duality pairing to give an
`elementary' construction of the functor $f^{!}$ for an arbitrary morphism $f : X \rightarrow Y$).

From the above discussion, we obtain the following computation of Hochschild chains of $\Coh(X)$.

\begin{lem}\label{lem:hhcoh}
$\CB{\Coh(X)} \simeq R\Gamma(X,\Delta^{!}\Delta_{*}\om^{\bullet}_{X}) \simeq R{\rm Hom}_{X \times X}(\Delta_{*}\O_{X}, \Delta_{*}\om^{\bullet}_{X})$
\end{lem}

\begin{proof}
In general, the Hochschild chains of a small dg category $\A$ can be computed as the composition of the functor $\Mod_{k} \rightarrow \Mod_{\A^{e}}$ sending $k$ to the diagonal bimodule $\A$, and the functor $\Mod_{\A^{e}} \rightarrow \Mod_{k}$ given as left Kan extension of the Yoneda pairing $\A^{e} \rightarrow \Mod_{k}$. Note that under the equivalence $\Mod_{\A^{op}} \otimes \Mod_{\A}
\rightarrow \Mod_{\A^{e}}$, the Yoneda pairing exhibits $\Mod_{\A^{op}}$ as the dual of $\Mod_{\A}$.
Letting $\A=\Coh(X)$ and using the duality pairing \ref{indpairing}, we obtain an identification of
$\Mod_{\Coh(X)^{op}}$ with $\Mod_{\Coh(X)} \simeq \Ind(X)$. Under this identification, the diagonal
bimodule corresponds to $\Delta_{*}\om^{\bullet}_{X}$, and we obtain the isomorphism $\CB{\Coh(X)} \simeq R\Gamma(X,\Delta^{!}\Delta_{*}\om^{\bullet}_{X})$. The isomorphism $R\Gamma(X,\Delta^{!}\Delta_{*}\om^{\bullet}_{X}) \simeq R{\rm Hom}_{X \times X}(\Delta_{*}\O_{X}, \Delta_{*}\om^{\bullet}_{X})$ follows from properness of the diagonal morphism, which gives adjointness between $\Delta_{*}$ and $\Delta^{!}$.
\end{proof}

The next lemma describes the functoriality of Lemma \ref{lem:hhcoh} for proper morphisms, and is simply a translation of Proposition \ref{prop:Phi_f}. 

\begin{lem}\label{lem:proppush}
Let $f: X \rightarrow Y$ be a proper morphism, so that we have a morphism $f_{*} : \Coh(X)
\rightarrow \Coh(Y)$ and an induced morphism $\CB{f_{*}}: \CB{\Coh(X)} \rightarrow \CB{\Coh(Y)}$.
Then under the identifications provided by Lemma \ref{lem:hhcoh}, the induced morphism $\CB{f_{*}}$ is given as the composition 
\begin{equation*}
\xymatrix{
R{\rm Hom}(\Delta_{*}\O_{X},\Delta_{*}\om^{\bullet}_{X}) \ar[r] & R{\rm Hom}(\Delta_{*}f_{*}\O_{X},\Delta_{*}f_{*}\om^{\bullet}_{X}) \ar[r] & R{\rm Hom}(\Delta_{*}\O_{Y},\Delta_{*}\om^{\bullet}_{Y}) 
}
\end{equation*}
where the first arrow is given by pushforward along $f \times f$ and the natural isomorphism $(f \times f)_{*} \Delta_{*} \simeq \Delta_{*} f_{*}$ and the second arrow is given by pre-composition with $\Delta_{*}$ applied to the natural map $\O_{Y} \rightarrow f_{*}\O_{Y}$ and post-composition
with $\Delta_{*}$ applied to the natural counit map $f_{*}\om^{\bullet}_{X} \simeq f_{*}f^{!}\om^{\bullet}_{Y} \rightarrow 
\om^{\bullet}_{Y}$.
\end{lem}

In some cases arising in classical algebraic geometry and representation theory, we have vanishing
of Hochschild homology above a certain degree. The following lemma describes the consequence of this
for negative cyclic homology and provides a means of constructing $S^1$-equivariance data for left
Calabi-Yau structures.


\begin{lem}\label{lem:hhtohcminus}
\begin{enumerate}
\item Let $\A$ be a small dg category and suppose $HH_{i}(\A)=0$ for $i > d$. Then the natural map $HC^{-}_{i}(\A) \rightarrow HH_{i}(\A)$ is an isomorphism for $i \geq d$.

\item Let $f: \A \rightarrow \B$ be a dg functor between small dg categories and suppose that $HH_{i}(\A)=0$ for $i>d-1$ and $HH_{i}(\B)=0$ for $i>d$. Then the map $HC^{-}_{d}(\B,\A) \rightarrow HH_{d}(\B,\A)$ is an isomorphism.
\end{enumerate} 
\end{lem}

\begin{proof}
Using $\CB{\A}[[u]]$ with differential $b+Bu$ as our model for the negative cyclic complex $\NC{\A}$
(cf. \cite{hoyois:cyclic}), we obtain a filtration $\cdots \subset u^{2}\NC{\A} \subset u\NC{\A} \subset
\NC{\A}$. Inspecting the associated spectral sequence immediately gives the first part of the lemma
and the second part follows immediately by looking at the map between long exact sequences of
negative cyclic and Hochschild homology. 

Equivalently, and very concretely, suppose a class in  $HC^{-}_{i}(\A)$ is represented by a cycle 
$c=\sum_{j=0}^{\infty} c_{i+2j}u^{j}$ in $\CB{\A}[[u]]$, so that we have relations $bc_{i}=0$ and $bc_{i+2j}+Bc_{i+2j-2}=0$ for $j>0$ in $\CB{\A}$. 

First, to show $HC^{-}_{i}(\A)=0$ for $i>d$, we must integrate the cycle $c$. The assumption that $HH_{i}(\A)=0$ for $i>d$ ensures that this can be done order by order. 

Next, let $i=d$. We want to show that the map $HC^{-}_{d}(\A) \rightarrow HH_{d}(\A)$ is an isomorphism. If $c=\sum_{j=0}^{\infty} c_{d+2j}u^{j}$ represents a class in the kernel of $HC^{-}_{d}(\A) \rightarrow HH_{d}(\A)$, then the constant coefficient of $c_{d}$ of $c$ must be $b$-exact, and so, replacing $c$ with a homologous cycle if necessary, we may assume that the constant coefficient of $c$ vanishes. But then we may write $c=\tilde{c}u$, where $\tilde{c}$  is a cycle of degree $d+2$ and hence null-homologous by the first part of our argument. Thus the map 
$HC^{-}_{d}(\A) \rightarrow HH_{d}(\A)$ is an injection. To see that it is a surjection, one inductively constructs
a lift of a $d$-cycle $c_{d}$ in $\CB{\A}$ to a $d$-cycle in $\CB{\A}[[u]]$ using the vanishing of $HH_{i}(\A)$ for $i>d$.

\end{proof}

\begin{rem}
Inspecting the above argument, one finds in addition that there is a short exact sequence $0 \rightarrow HH_{d-2}(\A) \rightarrow HC^{-}_{d-1}(\A) \rightarrow HH_{d-1}(\A) \rightarrow 0$, but we shall not need this.
\end{rem}

\subsubsection{Calabi-Yau schemes and anticanonical divisors} 

Recall that a scheme $X$ is said to be Cohen-Macaulay of dimension $d$ if $\om_{X}=\om^{\bullet}_{X}[-d]$ 
is a coherent sheaf (that is, lives in the heart of the $t$-structure on $\Ind(X)$) and is said to be
Gorenstein of dimension $d$ if $\om^{\bullet}_{X}[-d]$ is a line bundle. 

\begin{lem}\label{lem:hhcm}
Let $X$ be Cohen-Macaulay of dimension $d$. Then $HH_{d}(\Coh(X)) \cong H^{0}(X,\om_{X})$ and, for
$i > d$, we have $HH_{i}(\Coh(X)) = 0$. Furthermore, we have an isomorphism $HC^{-}_{d}(\Coh(X))
\simeq HH_{d}(\Coh(X))$.
\end{lem}
\begin{proof}
By Lemma \ref{lem:hhcoh}, $HH_{i}(\Coh(X)) \simeq Ext^{-i}(\Delta_{*}\O_{X},\Delta_{*}\om^{\bullet}_{X}) \simeq Ext^{d-i}(\Delta_{*}\O_{X},\Delta_{*}\om_{X})$, which clearly vanishes for $d-i<0$ since $\om_{X}=\om^{\bullet}_{X}[-d]$ is a sheaf. Moreover, when $i=d$, we have an isomorphism $HH_{d}(\Coh(X)) \simeq Ext^{0}(\Delta_{*}\O_{X},\Delta_{*}\om_{X}) \simeq H^{0}(X,\om_{X})$, since push-forward along a closed immersion is fully faithful in degree $0$ between sheaves.
\end{proof}

\begin{prop}\label{prop:cygor} Let $X$ be a Gorenstein scheme of dimension $d$. Then giving a left
	Calabi-Yau structure on $\Coh(X)$ is equivalent to giving a trivialization $\O_{X} \simeq
	\om_{X}$.
\end{prop}

\begin{proof}
By definition, a left Calabi-Yau structure of dimension $k$ on $\Coh(X)$ is given by a class $\theta \in HC^{-}_{k}(\Coh(X))$ such that the corresponding Hochschild class, viewed as a map
$\Delta_{*}\O_{X} \rightarrow \Delta_{*}\om^{\bullet}_{X}[-k]$, is an equivalence. Since $X$ is Gorenstein of dimension $d$, there is only one possibility for $k$, namely $k=d$, and the induced map $\Delta_{*}\O_{X} \rightarrow \Delta_{*}\om^{\bullet}_{X}[-d]=\Delta_{*}\om_{X}$ is an equivalence if and only the underlying map $\O_{X} \rightarrow \om_{X}$ is an equivalence. Thus a left Calabi-Yau structure on $\Coh(X)$ gives a trivialization $\O_{X} \simeq \om_{X}$. 

Conversely, given a trivilisation $\O_{X} \simeq \om_{X}=\om^{\bullet}_{X}[-d]$, apply $\Delta_{*}$ to obtain a class
$Ext^{-d}(\Delta_{*}\O_{X}, \Delta_{*}\om^{\bullet}_{X})$. By Lemma \ref{lem:hhcoh}, we have an isomorphism $HH_{d}(\Coh(X)) \simeq Ext^{-d}(\Delta_{*}\O_{X}, \Delta_{*}\om^{\bullet}_{X})$, and by Lemma \ref{lem:hhcm}, we have an isomorphism $HC^{-}_{d}(\Coh(X)) \simeq HH_{d}(\Coh(X))$. Altogether, a trivisalisation $\O_{X} \simeq \om_{X}$ gives a left Calabi-Yau structure on $\Coh(X)$.
\end{proof}

\begin{thm}\label{thm:antican}

Let $Y$ be a Gorenstein scheme of dimension $d$ and fix a section $s \in \om^{-1}_{Y}$ with zero scheme $i: X \hookrightarrow Y$ of dimension $d-1$. Then $\Coh(X)$ carries a canonical left Calabi-Yau structure of dimension $d-1$ and the pushforward functor $i_{*}: \Coh(X) \rightarrow \Coh(Y)$ carries a compatible canonical left Calabi-Yau structure of dimension $d$.
\end{thm}

\begin{proof}
The choice of section $s \in \om^{-1}_{Y}$ provides a cofibre sequence
\[
\O_{Y} \rightarrow i_{*}\O_{X} \rightarrow \om_{Y}[1]
\]
in $\Ind(Y)$ and in particular a null-homotopy of the composed map $\O_{Y} \rightarrow \om_{Y}[1]$. Applying $i^{!}$ to the above cofibre sequence and using the natural equivalences $i^{!}\om_{Y}[1] \simeq i^{!}\om^{\bullet}_{Y}[1-d] \simeq \om^{\bullet}_{X}[1-d] \simeq \om_{X}$ we obtain a cofibre sequence
\[
i^{!}\O_{Y} \rightarrow i^{!}i_{*}\O_{X} \rightarrow \om_{X}
\]
in $\Ind(X)$. Pre-composing $r: i^{!}i_{*}\O_{X} \rightarrow \om_{X}$ with the unit $u: \O_{X} \rightarrow i^{!}i_{*}\O_{X}$ of the adjunction, we obtain a morphism 
\[
\theta: \O_{X} \rightarrow \om_{X}
\]
and hence by Lemma \ref{lem:hhcoh} a class in $HC^{-}_{d-1}(\Coh(X)) \simeq \HH_{d-1}(\Coh(X))$. We claim that this class gives a left Calabi-Yau structure of dimension $d-1$ on $\Coh(X)$, which by Proposition \ref{prop:cygor} is equivalent to showing that the map $\O_{X} \rightarrow \om_{X}$ is an isomorphism.

To this end, note that since $i: X \hookrightarrow Y$ is by definition the inclusion of the zero-scheme of $s$, and the map $\om_{Y} \rightarrow \O_{Y}$ is the dual of $s: \O_{Y} \rightarrow \om^{-1}_{Y}$, the above cofibre sequence splits. Furthermore, the component of $\O_{X} \rightarrow i^{!}i_{*}\O_{X} \simeq i^{!}\O_{Y} \oplus \om_{X}$ projecting to $i^{!}\O_{Y}$ must be homotopic to zero since it is adjoint to a morphism $i_{*}\O_{X} \rightarrow \O_{Y}$, the latter being automatically zero since $X \subset Y$ is a proper closed subset. Finally, applying the functor $i_{*}$ and using the counit of the adjunction, we obtain a factorisation $i_{*}\O_{X} \rightarrow i_{*}\om_{X} \rightarrow i_{*}i^{!}i_{*}\O_{X} \rightarrow i_{*}\O_{X}$ of the identity morphism of $i_{*}\O_{X}$. Thus $i_{*}\theta: i_{*}\O_{X} \rightarrow i_{*}\om_{X}$ and hence $\theta: \O_{X} \rightarrow \om_{X}$ itself are invertible. 

To prove the second part of the theorem, it is enough by Lemma \ref{lem:hhtohcminus} to provide a null-homotopy of the composition
from top-left to bottom-right in the following diagram, such that the induced vertical arrows are equivalences:
\[
\xymatrix{\Delta_{*}\O_{Y} \ar[r] \ar@{..>}[d] & (i \times i)_{*}\Delta_{*}\O_{X}\simeq \Delta_{*}i_{*}\O_{X} \ar[d] \ar[r] & \Delta_{*}\om_{Y}[1] \ar@{..>}[d]\\
\Delta_{*}\O_{Y} \ar[r] & (i \times i)_{*}\Delta_{*}\om_{X} \simeq \Delta_{*}i_{*}\om_{X} \ar[r] & \Delta_{*}\om_{Y}[1]
}
\]
But in fact, we have already implicitly proved this before applying $\Delta_{*}$. Indeed, our constructions above provide a null-homotopy of $\O_{Y} \rightarrow i_{*}\O_{X} \rightarrow \om_{Y}[1]$ as well as a factorisation of $i_{*}\O_{X} \rightarrow \om_{Y}[1]$ into $i_{*}\O_{X} \rightarrow i_{*}\om_{X} \rightarrow \om_{Y}[1]$. Furthermore, we have just shown above that the middle vertical arrow is an equivalence, and the outer vertical arrows are nothing but identities.

\end{proof}

\subsection{Representation theory}

\subsubsection{$A_n$-quiver}

Let $k$ be a field and let $S$ denote the path algebra over $k$ of the $A_n$-quiver equipped with
the standard orientation. We label the vertices of the quiver by $1,2,\dots,n$ so that we have, for
every $1 \le i <n$, an arrow $\rho_{i,i+1}$ from $i$ to $i+1$. We denote the idempotent in $S$ corresponding to 
the vertex $i$ by $e_i$ and further set $e_0 = e_{n+1} = 0$. Our convention is to write concatenation of paths from left to right, so that for example $e_{i}\rho_{i,i+1}e_{i+1}=\rho_{i,i+1}$. For $0 \le i \le n$, we define $L_i \in
\Perf_S$ to be 
\[
			\dots \to 0 \to e_{i+1} S \to e_i S \to 0 \to \dots
\]
where the right $S$-module $e_i S$ is situated in degree $0$ and the arrow $e_{i+1} S \to e_i S$ is given by left multiplication with $\rho_{i,i+1}$.

\begin{rem} Note that the objects $L_1$, \dots, $L_n$ are cofibrant replacements of the simple
	$S$-modules. 
\end{rem}

We denote by $\amalg_{n+1} \underline{k}$ the free dg category on the set $\{0,1,2,\dots,n\}$.

\begin{thm}\label{thm:calabi-quiver}
	The dg functor
	\[
		f:\; \coprod_{n+1} \underline{k} \lra \Perf_S, \; i \mapsto L_i
	\]
	carries a natural left Calabi-Yau structure of dimension $1$.
\end{thm}

For the proof we use the notation $\A = \amalg_{n+1} \underline{k}$ and $\B = \Perf_S$. 

\begin{prop}\label{prop:quiverrel} There is an $S^1$-equivariant equivalence of complexes
	\[
		\CB{\B,\A} \simeq k[1]
	\]
	where the right-hand side is equipped with the trivial circle action.
\end{prop}
\begin{proof}
	The category $\B$ admits a semiorthogonal decomposition which makes it easy to compute
	Hochschild homology: the natural map $K_0(\B) \otimes k \to \CB{\B}$ is a quasi-isomorphism of
	complexes so that $\CB{\B} \cong \HH_0(\B)$ with a basis naturally given by isomorphism
	classes of simple $\B$-modules. Using this description, the map $\HH_0(\A) \to
	\HH_0(\B)$ can be canonically identified with the $k$-linear map $k^{n+1} \to k^n$ given by the matrix
	\[
		\left( \begin{array}{rrrrr} 
				-1 & 1 & 0 & \dots   & 0\\
				-1 & 0 & 1 & & 0\\		
				& & & \ddots  & \\
				-1 & 0 & 0 & \dots & 1 \end{array} \right).
	\]
	In particular, we have a short exact sequence
	\[
		0 \lra k \lra \HH_0(\A) \lra \HH_0(\B) \lra 0
	\]
	where all other Hochschild homology groups vanish. This implies the statement.
\end{proof}

\begin{rem}\label{rem:unique}
Proposition \ref{prop:quiverrel} implies that, up to rescaling, there is a unique class in
$H^{-1}(\CB{\B,\A})$ which further has a canonical negative cyclic lift. 
\end{rem}

\begin{proof}[Proof of Theorem \ref{thm:calabi-quiver}]
	The diagonal $\A$-bimodule $\A$ is given by 
	\[
		\A(i,j) \cong \begin{cases} k & \text{if $i = j$}\\
			0 & \text{if $i \ne j$}\end{cases}
	\]
	and its left dual $\A^!$ is therefore
	\[
		\A^!(i,j) \cong \begin{cases} k^* & \text{if $i = j$}\\
			0 & \text{if $i \ne j$}.\end{cases}
	\]
	We choose the identification
	\[
		\xi: \A^! \lra \A
	\]
	induced by $1^* \mapsto 1$ in each diagonal component. Let $F= f^{\op} \otimes f$ and
	consider the Quillen adjunction
	\[
		F_!: \Mod_{\Ae} \llra \Mod_{\Be}: F^*.
	\]
	We form the diagram 
	\begin{equation}\label{eq:part}
			\xymatrix{ \B^! \ar[r]^{c^!} & (F_!\A)^! \ar[d]^{F_!\xi}_{\simeq} & \\
		& F_!\A \ar[r]^c & \B}
	\end{equation}
	in the $\infty$-category $\D(\Mod_{\Be})$. To construct a left Calabi-Yau structure on $f$
	it suffices to provide a zero homotopy of the composite which exhibits $\B^!$ as the fiber of
	$c$. Indeed, such a zero homotopy gives, by definition, a nonzero class in
	$H^{-1}(\CB{\B,\A})$ which, by Remark \ref{rem:unique}, has a canonical negative cyclic lift. 

	The envelope $J = j^{\op} \otimes j$ of the Yoneda embedding $j: S \to \B$ induces an 
	equivalence of $\infty$-categories
	\[
		J^*: \D(\Mod_{\Be}) \lra \D(\Mod_{S^{\on{e}}}).
	\]
	It therefore suffices to construct a zero homotopy of the image of the diagram
	\eqref{eq:part} in $\D(\Mod_S)$ under $J^*$. This diagram can be explicitly computed as
	follows:
	The functor $j^*f_!$ is given by  
	\[
		 - \otimes_\A M : \Mod_\A \lra \Mod_S
	\]
	where $M \in \Mod^\A_S$ can be described as
	\[
		M = \oplus_i L_i. 
	\]
	The composite $J^*F_!$ is then given by 
	\[
		- \otimes_{\Ae} M^{\vee} \otimes_k M : \Mod_{\Ae} \lra \Mod_{\Se}
	\]
	with
	\[
		M^{\vee} = \oplus_i L_i^{\vee} \in \Mod^S_\A
	\]
	where $(-)\vee$ denotes the $S$-linear dual.  
	The bimodule $J^*F_!\A$ admits the formula
	\[
		M^{\vee} \otimes_\A M \cong  \oplus_i L_i^{\vee} \otimes L_i
	\]
	and the restricted counit map
	\[
		J^*(c): \oplus_i L_i^{\vee} \otimes L_i \lra S
	\]
	is given by the sum over the evaluation maps. The standard cofibrant
	replacement of $S$ as an $S$-bimodule is given by the complex
	\[
		\oplus_i S e_i \otimes e_{i+1} S \lra \oplus_i S e_i \otimes e_i S
	\]
	with $e_i \otimes e_{i+1} \mapsto \rho \otimes e_{i+1} - e_i \otimes \rho$ where $\rho$
	denotes the arrow in the quiver from $i$ to $i+1$. Here and below, the summation
	index $i$ runs from $0$ to $n$ where we declare $e_0 = e_{n+1} = 0$. In terms of this
	cofibrant replacement, the map $J^*(c)$ takes the explicit form
	\[
		\xymatrix{
		\oplus_i S e_i \otimes e_{i+1} S \ar[d] \ar[r]^-{\id} & \oplus_i S e_i \otimes e_{i+1} S\ar[d]\\
		\oplus_i (S e_i \otimes e_{i} S \oplus S e_{i+1} \otimes e_{i+1} S)  \ar[r]^-{\alpha}\ar[d] &
		\oplus_i S e_i \otimes e_{i} S\ar[d] \\
		{\underbrace{\oplus_i S e_{i+1} \otimes e_i S}_{J^*F_!\A}}\ar[r] &
		{\underbrace{0.}_{J^*\B}}
		}
	\]
	where the left (resp. right) column describes the object $J^*F_!\A$ (resp. $J^*\B$).
	The image of diagram \eqref{eq:part} in $\D(\Mod_{\Se})$, with $J^*F_!(\xi)$ left
	implicit, is induced by the strict short exact sequence of complexes of $S$-bimodules
	\[
		\xymatrix{
			0 \ar[r]\ar[d] & \oplus_i S e_i \otimes e_{i+1} S \ar[d] \ar[r]^-{\id} & \oplus_i S e_i \otimes
		e_{i+1} S\ar[d]\\
		\oplus_i S e_i \otimes e_i S \ar[r]^-{\alpha^{\vee}} \ar[d] & \oplus_i (S e_i \otimes e_{i} S \oplus S
		e_{i+1} \otimes e_{i+1} S) \ar[d]\ar[r]^-{\alpha} & \oplus_i S e_i \otimes e_{i} S \ar[d]\\
		{\underbrace{\oplus_i S e_{i+1} \otimes e_i S}_{J^*\B^!}} \ar[r]^-{\id} &
		{\underbrace{\oplus_i S e_{i+1} \otimes e_i S}_{J^*F_!\A}}\ar[r] &
		{\underbrace{0.}_{J^*\B}}
		}
	\]
	where the columns correspond to the objects $J^*\B^!$, $J^*F_!\A$, and $J^*\B$, respectively.
	We may therefore conclude the argument by choosing the trivial zero homotopy to obtain the required class in
	$H^{-1}(\CB{\B,\A})$.
\end{proof}

\subsection{Symplectic geometry}

The functors from Theorem \ref{thm:calabi-quiver} naturally fit into the algorithmic framework for
Fukaya categories as developed in \cite{seidel:fukaya}: Choose points $z_1, \dots, z_n \in \bC$
contained in the unit disk $D$, together with smooth paths $\alpha_i$ in $D$ from $1$ to each $z_i$
whose interior does not intersect, called {\em vanishing paths}. We assume that the ordering is
chosen so that the tangent directions to the paths  $\alpha_1, \dots, \alpha_n$ at $1$ are ordered
counter-clockwise. Further, choose loops $\gamma_i$ in $D$ based at $1$ which runs counter-clockwise
about $z_i$ and whose interior does not intersect any of the paths $\alpha_i$. The loops $\alpha_i$
freely generate the fundamental group of the disk $D$ punctured at $\{z_i\}$. We specify the
representation
\[
	\mu:\; \pi_1(D \setminus \{z_i\}, 1) \lra \Aut(\{0,1,\dots,n\}), \; \gamma_i \mapsto (0,i).
\]
Using the Riemann existence theorem, we obtain a polynomial $p(z)$ so that the corresponding
ramified cover 
\begin{equation}\label{eq:cover}
		p: \bC \to \bC
\end{equation}
has branch points $\{z_i\}$ and the monodromy representation of $p$ for a suitable identification
$p^{-1}(1) \cong \{x_0, \dots, x_n\}$ is $\mu$. 

For any symplectic Lefschetz fibration $q: X \to \bC$ with a chosen basis of vanishing paths, 
Seidel \cite{seidel:fukaya} has constructed an $A_{\infty}$-category $\Fuk(q)$ which interpolates 
between the Fukaya categories of the total space $X$ and the regular fiber $q^{-1}(1)$ in the 
following sense: There exist canonical functors
\[
	\Fuk(X) \overset{i}{\lra} \Fuk(q) \overset{g}{\lra} \Fuk(q^{-1}(1))
\]
where $i$ is fully faithful and $g$ can be described explicitly in terms of a certain directed category
construction.  The category $\Fuk(q)$ is generated by objects which correspond to Lagrangian vanishing
thimbles (equipped with extra data) and the functor $g$ is given by associating to such a thimble
its boundary.

In the context of \eqref{eq:cover}, this translates to the following. The Fukaya category
of the fiber $p^{-1}(1) = \{x_0,x_1,\dots,x_n\}$ is the perfect envelope of the free dg 
category on the set $x_0, x_1,\dots,x_n$ so that we have 
\[
	\Fuk(p^{-1}(1)) \simeq \amalg_{n+1} \Perf_k.
\] 
To each vanishing path $\alpha_i$ there is an associated vanishing thimble constructed as follows:
for every $i$, the two lifts of the path $\alpha_i$ along $p$ with starting point $x_0$ and $x_i$,
respectively, meet at their endpoint which is the unique ramification point lying over the branch
point $z_i$. The union of these two lifted paths forms a smooth path in $\bC$ with boundary
$\{x_0,x_i\}$, called the vanishing thimble associated with $\alpha_i$. The category $\Fuk(p)$ is
then defined as the directed subcategory of $\Fuk(p^{-1}(1))$ on the objects
\[
	x_0[1] \oplus x_1, x_0[1] \oplus x_2, \dots, x_0[1] \oplus x_n
\]
which correspond to the (graded) boundaries of the (graded) vanishing thimbles. We have $\Fuk(p)
\simeq \Perf_S$. The functor 
\[
	g: \Fuk(p) \lra \Fuk(p^{-1}(1))
\]
is adjoint to the functor in Theorem \ref{thm:calabi-quiver}. 

%

\section{Calabi-Yau cospans}
\label{sec:cobordism}

Let $X$ and $Y$ be oriented manifolds with boundaries $\partial X \cong S \amalg T$ and $\partial Y
\cong T \amalg U$ together with choices of collared neigborhoods of $T$. Then the pushout $X \amalg_S Y$ is
canonically an oriented manifold with boundary. The resulting composition law is the basic operation
of oriented cobordism. In this section, we establish a noncommutative analog of this construction
for functors of differential graded categories equipped with Calabi-Yau structures.   

\subsection{A noncommutative cotangent sequence}

We expect that there is a cotangent formalism which puts the following discussion into a formal
framework but we do not develop it here. Given a dg category $\A$, we propose to interpret the
diagonal $\A$-bimodule ${\A}$ as a shifted noncommutative cotangent complex $L_{\A}[1]$. Further,
given a functor $f: \A \to \B$, the exact sequence
\[
	\xymatrix{ \fib(c) \ar[r] \ar[d] & F_!(\A) \ar[d]^{c} \\
		0 \ar[r] & {\B} }
\]
of $\B$-bimodules should be interpreted as a relative cotangent sequence
\[
	\xymatrix{ L_{\B/\A} \ar[r] \ar[d] & L_{\A}[1] \otimes_{\A} \B \ar[d] \\
	0 \ar[r] & L_{\B}[1].}
\]
The following result fits naturally into this context and can be interpreted as the
exactness of the sequence 
\[
	\xymatrix{ L_{\A} \otimes_{\A} \B' \ar[r] \ar[d] & L_{\A'} \otimes_{\A} \B' \oplus L_{\B}
	\otimes_{\A} \B' \ar[d] \\
		   0 \ar[r] & L_{\B'} }
\]
for a given pushout square 
\[
			\xymatrix{ \A \ar[r]^f \ar[d]_g \ar[dr]^h & \B \ar[d]^i \\ \A' \ar[r]_j &
			\B' }
\]
of dg categories. This is in complete analogy to the behaviour of cotangent complexes with respect to
pushouts in the context of commutative differential graded algebras. 

\begin{thm}\label{thm:cotangent-pushout} Let 
	\begin{equation}\label{eq:psquare1}
			\xymatrix{ \A \ar[r]^f \ar[d]_g \ar[dr]^h & \B \ar[d]^i \\ \A' \ar[r]_j &
			\B' }
	\end{equation}
	be a pushout square in the $\infty$-category $\Lmo$. Then the corresponding square
	\begin{equation}\label{eq:psquare2}
			\xymatrix{ H_!({\A}) \ar[r] \ar[d] & I_!({\B}) \ar[d] \\
			J_!({\A'})
			\ar[r] & {\B'} }
	\end{equation}
	in the $\infty$-category of $\B'$-bimodules is a pushout square.
\end{thm}
\begin{proof}
	We use the model structure on $\Catdg(k)$ to reduce the general case to a special
	pushout square for which we can prove the statement by an explicit calculation involving bar
	constructions. We may assume that the square \eqref{eq:psquare1}
	is a pushout in the category $\Catdg(k)$ with $\A$ cofibrant and $f$, $g$
	cofibrations so that the square is a homotopy pushout. Further, by the small object
	argument, we may assume that the functor $g$ is a relative $I$-cell complex, where $I$ is
	the set of generating cofibrations in $\Catdg(k)$. Given a composite of pushout squares
	\[
		\xymatrix{ \A \ar[r] \ar[d] & \B \ar[d] \\ 
			\A' \ar[r] \ar[d] & \B' \ar[d] \\
			\A'' \ar[r] & \B'' }
	\]
	then if \eqref{eq:psquare2} is a pushout for the top and bottom squares it is a pushout for
	the exterior square. This observation generalizes to transfinite compositions of pushout
	squares so that we may assume that the functor $f$ in \eqref{eq:psquare1} is a pushout along
	a generating cofibration. In this situation, we have a composition of pushout squares of the form
	\begin{equation}\label{eq:generating}
			\xymatrix{ \X \ar[r]\ar[d] & \A \ar[r] \ar[d] & \B \ar[d] \\ \Y \ar[r] & \A'
			\ar[r] & \B' }
	\end{equation}
	where $\X \to \Y$ is one of the following morphisms:
	\begin{enumerate}
		\item $\emptyset \to \P$ where $\P$ denotes the $k$-linear dg
			category with one object and endomorphism ring $k$,
		\item $\S^{n-1} \to \D^{n}$, $n \in \ZZ$, where 
			\begin{itemize}
				\item $\S^{n-1}$ denotes the $k$-linear dg
			category with two objects $1$ and $2$, freely generated by a morphism $s:
			1 \to 2$ in degree $-(n-1)$ satisfying $d(s) = 0$, 
				\item $\D^n$ denotes the $k$-linear dg
			category with two objects $1$ and $2$, freely generated by a morphism $r: 1
			\to 2$ of degree $-n$ and a morphism $s: 1 \to 2$ of degree $-(n-1)$
			satisfying $d(r) = s$, 
				\item the functor is given by the apparent embedding of mapping complexes.
			\end{itemize}
	\end{enumerate}
	Note that if \eqref{eq:psquare2} is pushout (hence biCartesian) for the left-hand and the
	exterior squares in \eqref{eq:generating}, then \eqref{eq:psquare2} is a pushout for the
	right-hand square in \eqref{eq:generating} so that we can finally assume that
	\eqref{eq:psquare1} is of the form 
	\[
			\xymatrix{ \X \ar[r]\ar[d] & \B \ar[d] \\ \Y \ar[r]  & \B' }
	\]
	where $\X \to \Y$ one of (1) or (2) above. In case (1), it is immediate to verify that
	\eqref{eq:psquare2} is a pushout so that we are left with the square
	\[
		\xymatrix{ \S^{n-1} \ar[r]^f \ar[d] & \B \ar[d] \\ \D^{n-1} \ar[r]  & \B' }
	\]
	In this case, the morphism complex between objects $x,y$ in $\B'$ can be described
	explicitly as 
	\[
		\B'(x,y)  = \bigoplus_{n \ge 0} \B(x,x_1) \otimes k r  \otimes \B(x_2,x_1)
		\otimes k r \otimes \dots \otimes \B(x_2,y)
	\]
	where $n$ copies of $k r$ appear in the $n$th summand, $x_i = f(i)$, and the differential is 
	given by the Leibniz rule where, upon replacing $r$ by $d(r) = f(s)$, we also
	compose with the neighboring morphisms in $\B$ so that the level is decreased from
	$n$ to $n-1$. The square \eqref{eq:psquare2} being pushout is equivalent to the exactness of
	the sequence
	\[
		\xymatrix{ H_!({\S^{n-1}}) \ar[r] \ar[d] & J_!({\D^n}) \oplus  I_!({\B}) \ar[d] \\
		0 \ar[r] & {\B'} }
	\]
	which can be rewritten as
	\[
		\xymatrix{ \B \otimes_{\S^{n-1}}^L \B \ar[r] \ar[d] & \B \otimes_{\D^{n}}^L \B
		\oplus  \B \otimes_{\B'} \B \ar[d] \\
		0 \ar[r] & \B. }
	\]
	All terms of this sequence can be computed explicitly in terms of two-sided bar resolutions
	of the various diagonal bimodules. To show that the resulting sequence is exact it suffices
	to show that, for every pair of objects $(x,y)$ of $\B'$, the $k$-linear complex given by
	the totalization of the evaluation of the sequence at $(x,y)$ is contractible. This can be
	shown by filtering the totalization by the number of copies of $r$'s and providing exlicit
	contracting homotopies of the associated graded complexes. \end{proof}

\subsection{Composition of Calabi-Yau cospans}

Given {\em cospans} of differential graded categories
\[
	e \amalg f: \A \amalg \B \lra \X
\]
and 
\[
	g \amalg h: \B \amalg \C \lra \Y,
\]
we form a coherent diagram 
\begin{equation}\label{eq:composition}
		\xymatrix{
			& & \A \ar[d]^e  \\
			& \B \ar[r]^f \ar[d]_g & \X \ar[d]^-i \\
			\C \ar[r]_h & \Y \ar[r]_-j & \X \amalg_{\B} \Y }
\end{equation}
in $\Lmo$ where the bottom right square is a pushout. We call the functor
\[
	ie \amalg jh: \A \amalg \C \lra \X \amalg_{\B} \Y
\]
the {\em composite of the cospans $e \amalg f$ and $g \amalg h$}.
Applying $\CB{-}$ to \eqref{eq:composition}, we obtain a coherent $S^1$-equivariant diagram of complexes 
\begin{equation}\label{eq:cospans}
		\xymatrix{
			F\ar[r]\ar[d] & \CB{\X,\A \amalg \B}[-1] \ar[r]^-{\delta_{A}} \ar[d]_-{-\delta_{B}} & \CB{\A} \ar[d]\\
			\CB{\Y,\B \amalg \C}[-1]\ar[d]_-{-\delta_{\C}} \ar[r]^-{\delta_{\B}} & \CB{\B} \ar[r] \ar[d]
			& \CB{\X} \ar[d] \\
			\CB{\C} \ar[r] & \CB{\Y} \ar[r] & \CB{\X \amalg_{\B} \Y} }
\end{equation}
in which all squares, with possible exception of the bottom right square, are biCartesian. We give a more detailed description of \eqref{eq:cospans}:  
By definition, we have cofibre sequences 

\[
	\CB{\X,\A \amalg \B}[-1] \lra \CB{\A \amalg \B} \lra \CB{\X}
\]
and 
\[
	\CB{\Y,\B \amalg \C}[-1] \lra \CB{\B \amalg \C} \lra \CB{\Y}.
\]
Inverting the natural equivalences $\CB{\A} \oplus \CB{\B} \simeq \CB{\A \amalg \B}$ and $\CB{\B} \oplus \CB{\C} \simeq \CB{\B \amalg \C}$ and projecting to $\CB{\B}$, we obtain maps 
\[
	\delta_{\B} : \CB{\Y,\B \amalg \C}[-1] \lra \CB{\B}
\]
and 
\[
	\delta^{\prime}_{\B} : \CB{\X,\A \amalg \B}[-1] \lra \CB{\B}. 
\]
Defining $F$ to be the fibre (or $F[1]$ to be the cofibre) of the difference $\delta_{\B}-\delta^{\prime}_{\B}$, we obtain a cofibre sequence 
\[
	\CB{\Y,\B \amalg \C}[-1] \oplus \CB{\X,\A \amalg \B}[-1] \lra \CB{\B} \lra F[1].
\]
and an equivalence $F[1] \simeq  \CB{\X,\A \amalg \B} \times_{\CB{\B}[1]} \CB{\Y,\B \amalg \C}$.
By construction, the compositions $\CB{\Y,\B \amalg \C}[-1] \lra \CB{\B} \lra \CB{\Y}$ and $\CB{\X,\A \amalg \B}[-1] \lra \CB{\B} \lra \CB{\X}$ are endowed with null-homotopies, and so the composition $\CB{\Y,\B \amalg \C}[-1] \oplus \CB{\X,\A \amalg \B}[-1] \lra \CB{\B} \lra \CB{\X \amalg_{\B} \Y}$ is endowed with a null-homotopy. We therefore obtain an induced $S^1$-equivariant morphism
\[
\chi:\; \CB{\X,\A \amalg \B} \times_{\CB{\B}[1]} \CB{\Y,\B \amalg \C} \simeq F[1] \lra  \CB{\X \amalg_{\B} \Y}
\]
which corresponds to the outer rectangle in \eqref{eq:cospans}.
The following result allows us to transport left Calabi-Yau structures along compositions of
cospans.

\begin{thm}\label{thm:cospan} 
Let 
\[
	e \amalg f: \A \amalg \B \lra \X
\]
and 
\[
	g \amalg h: \B \amalg \C \lra \Y,
\]
be functors of smooth dg categories. Let 
\[
 \sigma \in \NC{\X,\A \amalg \B} \times_{\NC{\B}[1]} \NC{\Y,\B \amalg \C}
\]
so that the projections $\pi_1(\sigma)$ and $\pi_2(\sigma)$ define left Calabi-Yau structures on $e
\amalg f$ and $g \amalg h$, respectively. Then $\chi(\sigma)$ defines a left Calabi-Yau structure on
the composite $ie \amalg jh$.
\end{thm}
\begin{proof} 
	The Calabi-Yau structure on $e \amalg f$ induces an identification of exact sequences
	\[
		\xymatrix{
			{\X}^! \ar[r] \ar[d]^{\simeq} & (E \amalg F)_!({\A \amalg \B})^! \ar[r]
			\ar[d]^{\simeq} & \cof_{\X} \ar[d]^{\simeq} \\
			\fib_{\X} \ar[r] & (E \amalg F)_! ({\A \amalg \B}) \ar[r] & {\X}
		}
	\]
	of $\X$-bimodules.
	We can reinterpret this as an equivalence between the biCartesian squares
	\[
		\xymatrix{
			\fib_{\X} \ar[r] \ar[d] & E_!{\A} \ar[d] \\
			F_!{\B} \ar[r] & {\X}  }
	\]
	and
	\[
		\xymatrix{
			({\X})^! \ar[r] \ar[d] & (E_! {\A})^! \ar[d] \\
			(F_!{\B})^! \ar[r] & \cof_{\X}.}
	\]
	An analogous statement holds for the Calabi-Yau structure on $g \amalg h$. Both Calabi-Yau
	structures combined yield an equivalence of the diagrams
	\begin{equation}\label{eq:diag1}
			\xymatrix{
				& I_!(\fib_{\X}) \ar[r]\ar[d]  & I_!(E_! {\A}) \ar[d]\\
				J_! \fib_{\Y} \ar[r]\ar[d] & M_!{\B} \ar[r]\ar[d] & I_! {\X} \\
				J_!(H_! {\C}) \ar[r] & J_! {\Y} & }
	\end{equation}
	and 
	\begin{equation}\label{eq:diag2}
			\xymatrix{
				& I_!({\X}^!) \ar[r]\ar[d]  & I_!(E_! {\A}^!) \ar[d]\\
				J_!({\Y}^!) \ar[r]\ar[d] & M_!({\B}^!) \ar[r]\ar[d] & I_! \cof_{\X} \\
				J_!(H_! {\C}^!) \ar[r] & J_! \cof_{\Y} & }
	\end{equation}
	of $\Z$-bimodules, where $\Z=\X \amalg_{\B} \Y$ and $M_{!} \simeq I_{!}F_{!} \simeq J_{!}G_{!}$. 
	By Theorem \ref{thm:cotangent-pushout}, we have a pushout square
	\[
		\xymatrix{
		M_!{\B} \ar[r]\ar[d] & I_! {\X} \ar[d]\\
		J_! {\Y} \ar[r] &  {\Z}. }
	\]
	It follows that we may complete \eqref{eq:diag1} and \eqref{eq:diag2} by forming pullbacks and pushouts to
	obtain, an equivalence between the diagrams
	\[
		\xymatrix{
			\fib_{\Z} \ar[r]\ar[d] & I_!(\fib_{\X}) \ar[r]\ar[d]  & I_!(E_! {\A}) \ar[d]\\
			J_! \fib_{\Y} \ar[r]\ar[d] & M_!{\B} \ar[r]\ar[d] & I_! {\X} \ar[d]\\
			J_!(H_! {\C}) \ar[r]       & J_! {\Y} \ar[r] &  {\Z} }
	\]
	and 
	\[
		\xymatrix{
			{\Z}^! \ar[r]\ar[d] & I_!({\X}^!) \ar[r]\ar[d]  & I_!(E_! {\A}^!) \ar[d]\\
			J_!({\Y}^!) \ar[r]\ar[d] & M_!({\B}^!) \ar[r]\ar[d] & I_! \cof_{\X}\ar[d] \\
			J_!(H_! {\C}^!) \ar[r] & J_! \cof_{\Y} \ar[r] & \cof_{\Z}.}
	\]
	Here, we use Lemma \ref{lem:delta} to obtain that, for a dg functor $r: \S \to \T$ with $\S$
	smooth, we have $R_!({\S}^!) \simeq (R_! {\S})^!$.
	We restrict this equivalence to the exterior biCartesian square to obtain the desired
	equivalence showing the nondegeneracy of the boundary structure on $ie \amalg jh$. 
\end{proof}

\section{Applications}
\label{sec:app}

\subsection{Localization}

Let $f: \A \to \B$ be a functor of dg categories which carries a relative Calabi-Yau
structure. Then the choice of negative cocycle on $\A$ also defines a relative Calabi-Yau structure
on the zero functor $\A \to 0$. The pushout square
\[
	\xymatrix{\A \ar[r]\ar[d] & \B \ar[d]\\
	0 \ar[r] & \B/\A }
\]
in $\Lmo$ can be interpreted as a composition of the Calabi-Yau cospans
\[
	0 \amalg \A \lra \B
\]
and
\[
	\A \amalg 0 \lra 0.
\]
Therefore, the $n$-dimensional left Calabi-Yau structure on $f$ induces a canonical morphism
\begin{equation}\label{eq:cofiber}
		k[n] \lra \NC{\B,\A}.
\end{equation}
The following is therefore an immediate corollary of Theorem \ref{thm:cospan}:

\begin{cor}\label{cor:loc} Let $f: \A \to \B$ be a functor of dg categories which carries a left Calabi-Yau
	structure. Then the negative cyclic cocycle \eqref{eq:cofiber} defines a left Calabi-Yau
	structure on the cofiber $\B /\A$.
\end{cor}

\begin{ex}
	We have a pushout square of topological $\ZZ$-graded Fukaya categories (cf. Section
	\ref{sec:fuk} below) 
\[
\begin{tikzpicture}[auto]
	\node (A11)  {$F(\;$\tikz[baseline=(current bounding box.center),scale=0.7]{
			\filldraw[black] (-.2,0) circle(1.5pt);
			\filldraw[black] (.2,0) circle(1.5pt);
			\draw (-.2,-.5) -- (-.2,0) -- (-.2,.5);
			\draw (.2,-.5) -- (.2,0) -- (.2,.5);
		}$\;)$};

		\node (A12) [right= of A11] {$F(\;$\tikz[baseline=(current bounding box.center),scale=0.7]{
			  \filldraw[black] (270:.4) circle(1.5pt);
			  \draw (0,0) circle(0.4);
			  \draw (270:.4) -- (270:.9);
			  \draw (90:.1) -- (270:.4);
		}$\;)$}; 

		\node (A21) [below= of A11] {$F(\;$\tikz[baseline=(current bounding box.center),scale=0.7]{
		  \filldraw[black] (-.2,.4) circle(1.5pt);
		  \draw (-.2,0) -- (-.2,.4);
		  \filldraw[black] (.2,0) circle(1.5pt);
		  \draw (.2,0) -- (.2,.4);
		}$\;)$};

		\node (A22) at (A12 |- A21) {$F(\;$\tikz[baseline=(current bounding box.center),scale=0.7]{
			  \filldraw[black] (270:.4) circle(1.5pt);
			  \filldraw[black] (90:.1) circle(1.5pt);
			  \filldraw[black] (270:.9) circle(1.5pt);
			  \draw (0,0) circle(0.4);
			  \draw (270:.4) -- (270:.9);
			  \draw (90:.1) -- (270:.4);
		}$\;)$}; 

		\draw [->] (A11) -- (A12);
		\draw [->] (A11) -- (A21);
		\draw [->] (A21) -- (A22);
		\draw [->] (A12) -- (A22);
	\end{tikzpicture}
\]
in $\Lmo$ which is equivalent to the pushout square
\[
\begin{tikzpicture}[auto]
	\node (A11)  {$\Coh(\{0\} \amalg \{\infty\})$};

	\node (A12) [right= of A11] {$\Coh(\PP^1)$}; 

		\node (A21) [below= of A11] {$0$};

		\node (A22) at (A12 |- A21) {$\Coh(\AA^1 \setminus \{0\})$}; 

		\draw [->] (A11) -- (A12);
		\draw [->] (A11) -- (A21);
		\draw [->] (A21) -- (A22);
		\draw [->] (A12) -- (A22);
	\end{tikzpicture}
\]
where we write $\Coh(-)$ for the derived dg category $D^b(\coh(-))$. The cofiber 
\[
	\Coh(\AA^1 \setminus \{0\})
\]
carries a natural induced left Calabi-Yau structure which, in this example, can be explicitly verified.
\end{ex}

\begin{ex} Let $X$ be a compact oriented manifold with boundary $\partial X$. Then we have a pushout square
	\[
		\xymatrix{
			\L(\partial X)\ar[d] \ar[r] & \L(X) \ar[d]\\
			0 \ar[r] & \L(X'),
		}
	\]
	where $X'$ is the closed oriented manifold given by the complement of the union of those
	connected components of $X$ which have nonempty boundary. The absolute left Calabi-Yau
	structure on $\L(X')$ corresponding to the orientation of $X'$ agrees with the one implied by 
	Corollary \ref{cor:loc}.
\end{ex}


\subsection{Topological Fukaya categories}
\label{sec:fuk}

According to Kontsevich \cite{kontsevich:symplectic}, the Fukaya category of a Stein manifold,
equipped with suitable extra data, can be described as the global sections of a cosheaf of
constructible dg categories on a Lagrangian spine. As stressed in \cite{kontsevich:symplectic}, to
obtain a version of the Fukaya category of finite type, it is crucial that the construction involves
a cosheaf and not a sheaf: by a formal argument, a finite colimit of finite type dg categories is of
finite type while finite limits do {\em not} inherit the finite type property.

Various constructions have been given for the $2$-dimensional case of a noncompact Riemann surface.
We recall the context of \cite{dyckerhoff:a1} where the topological Fukaya category of a stable
marked surface $(S,M)$ is defined (based on the ideas of \cite{dk:triangulated}). To obtain an
intrinsically defined $\ZZ$-graded version of the Fukaya category additional structure
on the surface must be given: we assume that the punctured surface $S \setminus M$ comes equipped
with a framing. The topological Fukaya category of $(S,M)$ can then be constructed combinatorially 
via the formalism of paracyclic $2$-Segal objects.

\subsubsection{State sums for coparacyclic objects}

The paracyclic category $\Lambda_{\infty}$ has objects $\cn$, $n \ge 0$, labelled by the natural
numbers. A morphism from $\cm$ to $\cn$ is a map $\varphi: \ZZ \to \ZZ$ preserving $\le$ and
satisfying $\varphi(z+m+1) = \varphi(z) + n+1$. 

Let $(S,M)$ be a stable marked surface with framing on $S \setminus M$ and let $\Gamma \subset S
\setminus M$ be a spanning graph. As explained in detail in \cite{dk:crossed}, the framing defines
combinatorial data on the graph $\Gamma$ which can be formulated as a functor
\[
	\delta: I(\Gamma)^{\op} \lra \Lambda_{\infty}.
\]
Given a coparacyclic object $X: \N(\Lambda_{\infty}) \lra \C$ with values in an $\infty$-category
$\C$ with colimits, we can then define the {\em state sum of $X$ on $\Gamma$} 
\[
	X(\Gamma) = \colim_{\N(I(\Gamma)^{\op})} X \circ \delta.
\]
The object $X$ is called {\em $2$-Segal} if $X$ maps edge contractions to equivalences in $\C$. As
shown in \cite{dk:crossed}, any coparacyclic $2$-Segal object defines by means of the state sum an
invariant $X(S,M)$ of the framed surface $(S,M)$ equipped with an action of the framed mapping class
group. 

\subsubsection{Left Calabi-Yau structures on topological Fukaya categories of framed surfaces}

We apply the above state sum construction to the coparacyclic $2$-Segal dg category
\[
	F: \N(\Lambda_{\infty}) \lra \Lmo,\; \cn \mapsto \MF^{\ZZ}(k[z], z^{n+1})
\]
where we refer to \cite{dyckerhoff:a1} for details. The resulting dg category $F(S,M)$ is called the
{\em topological Fukaya category of $(S,M)$}. For every component of the punctured boundary
$\partial S \setminus M$, we obtain a functor $\underline{k} \lra F(S,M)$. Summing over these functors, we
obtain the {\em boundary functor}
\[
	\coprod_{\pi_0(\partial S \setminus M)} \underline{k} \lra F(S,M)
\]

\begin{exa}
	Let $S \subset \mathbb C$ be the closed unit disk equipped with the standard framing. Let $M
	\subset \partial S$ be the subset of $(n+1)$st roots of unity. Then the boundary functor
	can be identified with the functor $f$ from Theorem \ref{thm:calabi-quiver}. In particular,
	it carries a left Calabi-Yau structure of dimension $1$.
\end{exa}

\begin{exa}\label{exa:sphere} Let $S$ be the $2$-sphere with $M$ consisting of two points. Suppose that $S \setminus
	M$ is equipped with a framing with winding number $n \in \ZZ$. Then the topological Fukaya
	category $F(S,M)$ is Morita equivalent to the dga $k[t,t^{-1}]$ with $|t| = 2n$ and zero
	differential.
\end{exa}

\begin{thm}\label{thm:fukcy} Let $(S,M)$ be a stable marked surface with framing on $S \setminus M$. Then the
	boundary functor
	\[
		\coprod_{\pi_0(\partial S \setminus M)} \underline{k} \lra F(S,M)
	\]
	carries a left Calabi-Yau structure of dimension $1$.
\end{thm}
\begin{proof}
	We choose a spanning graph $\Gamma \subset S \setminus M$ with one vertex. Then the
	topological Fukaya category $F(S,M)$ can be described as a pushout
	\[
		\xymatrix{
						& \amalg_{X} \underline{k} \ar[d]\\
			\amalg_{Y} \underline{k} \ar[d]_g \ar[r] & \Perf_S \ar[d]\\
			\amalg_{Z} \Perf_k \ar[r] & F(S,M)
		}
	\]
	where $X$ is the set of external half-edges of $\Gamma$, $Z$ the set of loops, and
	$Y$ the set of half-edges appearing in loops. The functor
	\[
		f: \coprod_{X} \underline{k} \amalg \coprod_{Y} \underline{k}\lra \Perf_S
	\]
	is the functor from Theorem \ref{thm:calabi-quiver}. The component of $g$ corresponding to a
	loop $l$ in $\Gamma$ is given by
	\begin{equation}\label{eq:formula}
			\underline k \amalg \underline k \lra \Perf_k 
	\end{equation}
	where the first object maps to $k$ and the second object maps to $k[2p]$ where $p$ is the
	winding number around the loop $l$. Here, $g$ is obtained by applying \eqref{eq:formula} to
	the pair of objects $\amalg_{Y} \underline{k}$ which correspond to half-edges comprising
	the loop $l$, and mapping all other components to zero. The functor $f$
	carries a left Calabi-Yau structure by Theorem \ref{thm:calabi-quiver}. A similar
	computation shows that $g$ carries a left Calabi-Yau structure which is compatible with the
	one on $f$. Therefore, Theorem \ref{thm:cospan} implies the result.
\end{proof}

\begin{rem} Let $(S,M)$ be as in Theorem \ref{thm:fukcy} and assume in addition that $S$ has no
	boundary. Then the topological Fukaya category $F(S,M)$ carries an absolute left Calabi-Yau
	structure. This was already stated without proof in \cite{kontsevich:symplectic}.
\end{rem}

\begin{rem} Note that Theorem \ref{thm:fukcy} does not address the question whether the left
	Calabi-Yau structure is canonical. 
\end{rem}

\begin{exa} Let $S$ be a torus with $M$ consisting of one marked point and with the standard framing. 
	Then we have a Morita equivalence
	\[
		F(S,M) \simeq \Coh(X)
	\]
	where $X$ is a nodal cubic curve in $\mathbb P^2$. Since $\partial S = \emptyset$,
	Theorem \ref{thm:fukcy} provides an absolute left Calabi-Yau structure on $\Coh(X)$,
	recovering a special instance of Section \ref{sec:algebraic}.
\end{exa}

{\bibliographystyle{halpha} 
\bibliography{refs} 

\end{document}